\documentclass[a4paper,11pt, reqno]{amsart}
\usepackage{geometry}
\usepackage{fullpage}
\usepackage[parfill]{parskip}   
\usepackage{graphicx}
\usepackage{amssymb}
\usepackage{amsmath}
\usepackage{epstopdf}
\usepackage{amsmath}
\usepackage{cite}

\usepackage{amsmath,amsthm,amscd,amssymb}
\usepackage{latexsym}
\usepackage[colorlinks,citecolor=red,pagebackref,hypertexnames=false]{hyperref}
 
\numberwithin{equation}{section}

\theoremstyle{plain}
\newtheorem{theorem}{Theorem}[section]
\newtheorem{Def}[theorem]{Definition}

\newtheorem{proposition}[theorem]{Proposition}

\theoremstyle{definition}

\theoremstyle{remark}
\newtheorem{remark}[theorem]{Remark}

\newtheorem{case[theorem]}{Case}

\def \N{{\mathbb N}}

\def \Z{{\mathbb Z}}

\def \T{{\mathbb T}}

\def\supp{\hbox{supp\,}}
\def\norm#1.#2.{\lVert#1\rVert_{#2}}

\title[Localization operators on discrete modulation spaces]{Localization operators on discrete modulation spaces}

\author{Aparajita Dasgupta}

\author{Anirudha Poria}

\thanks{Research supported by Core Research Grant (RP03890G), Science and Engineering Research Board (SERB), DST, India.}

\address{Department of Mathematics, Indian Institute of Technology Delhi, New Delhi 110016, India}
\email{adasgupta@maths.iitd.ac.in}

\address{Department of Mathematics, Indian Institute of Technology Delhi, New Delhi 110016, India}
\address{Department of Mathematics, SRM Institute of Science and Technology, Kattankulathur 603203, Tamil Nadu, India}
\email{anirudhamath@gmail.com, anirudhp@srmist.edu.in}



\date{\today}
\begin{document}
\maketitle

\begin{abstract} 
In this paper, we study a class of pseudo-differential operators known as time-frequency localization operators on $\Z^n$, which depend on a symbol $\varsigma$ and two windows functions $g_1$ and $g_2$. We define the short-time Fourier transform on $ \Z^n \times \T^n $ and modulation spaces on $\Z^n$, and present some basic properties. Then, we use modulation spaces on $\Z^n \times \T^n$ as appropriate classes for symbols, and study the boundedness and compactness of the localization operators on modulation spaces on $\Z^n$. Then, we show that these operators are in the Schatten--von Neumann class. Also, we obtain the relation between the Landau--Pollak--Slepian type operator and the localization operator on $\Z^n$. Finally, under suitable conditions on the symbols, we prove that the localization operators are paracommutators, paraproducts and Fourier multipliers.
\end{abstract}

\section{Introduction}  
Time-frequency localization operators are a mathematical tool to analyze functions in different regions in the time-frequency plane. They can be viewed as transformations that modify the content of a function simultaneously in time and frequency and reconstruct a filtered signal. The localization operators were introduced and studied by Daubechies \cite{dau88, dau90, pau88}, Ramanathan and Topiwala \cite{ram93}, and extensively investigated in \cite{fei02, won99, won02}. This class of operators occurs in various branches of applied and pure mathematics and has been studied by many authors. Localization operators are recognized as an important new mathematical tool and have found many applications to the theory of differential equations, quantum mechanics, time-frequency analysis, and signal processing (see \cite{cor03, mar02, now02, gro01, ram93, won02}). They are also known as Toeplitz operators, wave packets, anti-Wick operators, or Gabor multipliers (see \cite{now02, ber71, cor78, fei02}). For a detailed study of the theory of localization operators, we refer to the series of papers of Wong \cite{bog04, liu07, won2001, won2002, won2003}, and also the book of Wong \cite{won02}. In this paper, we attempt to study the localization operators on modulation spaces on $\Z^n$. Also, we show that localization operators with separable symbols are paracommutators. Paracommutators are bilinear pseudo-differential operators that contain Hankel operator, Toeplitz operator, and many other operators as special cases. The motivation to study the localization operators as paracommutators is to generalize many well-known operators. In particular, if the symbol is dependent on one variable, we show that localization operators can be viewed as paraproducts and Fourier multipliers.

Time-frequency localization operators were defined using the Schr\"odinger representation and the short-time Fourier transform, which suggests studying these operators as a part of time-frequency analysis. Modulation spaces were used as the appropriate function spaces for understanding these operators, as these spaces are associated to the short-time Fourier transform. Since localization operators are known as a class of pseudo-differential operators, recent works in pseudo-differential operators on $\Z^n$ (see \cite{ bot20}) motivated us to study the localization operator on $\Z^n$. Here, we define the modulation space on $\Z^n$ using the short-time Fourier transform on $\Z^n \times \T^n$. Our main aim in this paper is to expose and study the boundedness and compactness of the localization operators on $\Z^n$ under suitable conditions on symbols and windows and show that these operators are in the Schatten--von Neumann class. Then, we provide a connection between the Landau--Pollak--Slepian type operator and the localization operator on $\Z^n$. Finally, we show that under suitable conditions on the symbols the localization operators are paracommutators, paraproducts and Fourier multipliers.

For the study of the localization operator on $\mathbb{Z}^n$, what kind of spaces should be considered for the symbol? In the case of the localization operator on $\mathbb{R}^{n}$,  the symbol is a function on $\mathbb{R}^{n} \times \mathbb{R}^{n} $. Recent works in pseudo-differential operators on topological groups $G$ suggest that the correct phase space to work in is $G \times \widehat{G}$, where $\widehat{G}$ is the dual group of $G$. Since the dual group of $\mathbb{R}^{n}$ is the same as $\mathbb{R}^{n}$, the phase space on which symbols are defined is $\mathbb{R}^{n} \times \mathbb{R}^{n}$. For the group $\mathbb{Z}^n$, the dual group is $\T^n$ and the phase space $G \times \widehat{G}$ is $\mathbb{Z}^n \times \T^n$. In this paper, we consider the symbol as a function on $\mathbb{Z}^n \times \T^n$ and study the localization operator on $\mathbb{Z}^n$.

For $1 \leq p<\infty$, the set of all measurable functions $F$ on $\mathbb{Z}^n$ such that
\[ \|F\|_{\ell^{p}(\mathbb{Z}^n)}^{p}=\sum_{k \in \mathbb{Z}^n}|F(k)|^{p}<\infty \]    
is denoted by $\ell^{p}(\mathbb{Z}^n)$. We define $L^{p}(\T^n)$ to be the set of all measurable functions $f$ on $\T^n$ for which
\[ \|f\|_{L^{p} (\T^n )}^{p}= \int_{\T^n} |f(w)|^{p} \;dw < \infty. \]
Next, we define the Fourier transform $\mathcal{F}_{\mathbb{Z}^n} F$ of $F \in \ell^{1}(\mathbb{Z}^n)$ to be the function on $\T^n$ by
\[ \left(\mathcal{F}_{\mathbb{Z}^n} F \right)(w)=\sum_{k \in \mathbb{Z}^n} e^{-2 \pi i k \cdot w} F(k), \quad w \in \T^n. \]
Let $f$ be a function on $\T^n$, then we define the Fourier transform $\mathcal{F}_{\T^n} f$ of $f$ to be the function on $\mathbb{Z}^n$ by
\[ \left(\mathcal{F}_{\T^n} f\right)(k)= \int_{\T^n} e^{2 \pi i k \cdot w} f(w) \;dw, \quad k \in \mathbb{Z}^n. \]
Note that $\mathcal{F}_{\mathbb{Z}^n}: \ell^{2}(\mathbb{Z}^n) \rightarrow L^{2} (\T^n)$ is a surjective isomorphism. 
Also, $\mathcal{F}_{\mathbb{Z}^n}=\mathcal{F}_{\T^n}^{-1}=\mathcal{F}_{\T^n}^{*} $ and $\left\|\mathcal{F}_{\mathbb{Z}^n} F\right\|_{L^{2}\left(\T^n\right)}=\|F\|_{\ell^{2}(\mathbb{Z}^n)},$ $F \in \ell^{2}(\mathbb{Z}^n)$. 

For $1 \leq p<\infty$, we define $L^{p}\left(\mathbb{Z}^n \times \T^n \right)$ to be the space of all measurable functions $H$ on $\mathbb{Z}^n \times \T^n$ such that
\[ \|H\|_{L^{p}\left(\mathbb{Z}^n \times \T^n \right)}^{p}= \sum_{k \in \mathbb{Z}^n} \int_{\T^n}  |H(k, w)|^{p} \;dw < \infty .\]

The paper is organized as follows. In Section \ref{sec2},
we define the short-time Fourier transform on $ \Z^n \times \T^n $ and study some basic properties such as orthogonality relation, Plancherel's formula, and inversion formula. In Section \ref{sec3}, we define the modulation spaces on $\Z^n$ and present some basic properties. In Section \ref{sec4}, we define and study the localization operator on $\Z^n$. We use modulation spaces on $\Z^n \times \T^n$ as appropriate classes for symbols, and study the boundedness and compactness of the localization operators on modulation spaces on $\Z^n$. Also, we show that these operators are in the Schatten--von Neumann class. In Section \ref{sec5}, we define the Landau--Pollak--Slepian type operator on $\Z^n$ and show that this operator is a localization operator on $\Z^n$. In Section \ref{sec678}, we study the localization operators for different types of symbols. First, in Subsection \ref{sec6}, we show that the localization operators with separable symbols are paracommutators. Then, in Subsection \ref{sec7}, we prove that if the symbol is independent of the second variable, i.e., a function on $\Z^n$ only, then the localization operator can be expressed in terms of a paraproduct. Finally, in Subsection \ref{sec8}, we show that if the symbol is independent of the first variable, i.e., a function on $\T^n$ only, then the localization operator is a Fourier multiplier.

\section{Short-time Fourier transform on $\Z^n \times \T^n$}\label{sec2}

The Schwartz space $\mathcal{S}\left(\mathbb{Z}^{n}\right)$ on $\mathbb{Z}^{n}$ is the space of rapidly decreasing functions $g: \mathbb{Z}^{n} \rightarrow \mathbb{C}$, i.e. $g \in \mathcal{S}\left(\mathbb{Z}^{n}\right)$ if for any $M<\infty$ there exists a constant $C_{g, M}$ such that 
$$|g(k)| \leq C_{g, M} (1+|k|)^{-M}, \quad \text { for all } k \in \mathbb{Z}^{n}.$$ 
The topology on $\mathcal{S}\left(\mathbb{Z}^{n}\right)$ is given by the seminorms $p_{j}$, where $j \in \mathbb{N}_{0}= \N \cup \{0\}$ and $p_{j}(g):=\sup\limits_{k \in \mathbb{Z}^{n}}(1+|k|)^{j}|g(k)|$. The space $\mathcal{S}^{\prime}\left(\mathbb{Z}^{n}\right)$ of all continuous linear functionals on $\mathcal{S}\left(\mathbb{Z}^{n}\right)$ is called the space of tempered distributions. 

Fix $k \in \Z^n$, $w \in \T^n$ and $f \in \ell^2(\Z^n)$. For $m \in \Z^n$, the translation operator $T_k$ is defined by $T_k f(m)=f(m-k)$ and the modulation operator $M_w$ is defined by $M_w f(m)=e^{2 \pi i w \cdot m} f(m)$. Let $g \in \mathcal{S}\left(\mathbb{Z}^{n}\right)$ be a fixed window function. Then, the short-time Fourier transform (STFT) of a function $f \in \mathcal{S}^{\prime}\left(\mathbb{Z}^{n}\right)$ with respect to $g$ is defined to be the function on $\Z^n \times \T^n$ given by
\[ V_g f (m, w)= \left\langle f, M_w T_m g \right\rangle =\sum_{k \in \Z^n} f(k) \overline{M_w T_m g (k)} =\sum_{k \in \Z^n} f(k) \overline{g(k-m)} e^{-2 \pi i w \cdot k}.\]
For $k \in \Z^n$, we define $\tilde{g}(k)=g(-k)$. Then, we can write $V_g f$ as a convolution on $\Z^n$
\[ V_g f (m, w)= e^{-2 \pi i w \cdot m} \left( f * M_w \overline{\tilde{g}} \right)(m). \]
Next, we present some basic properties of the STFT on $\Z^n \times \T^n$.
\begin{proposition}
\begin{itemize}
\item[(1)] $($Orthogonality relation$)$  For every $f_1, f_2, g_1, g_2 \in \ell^2(\Z^n)$, we have
\begin{equation}\label{eq18}
\left\langle V_{g_1} f_1, V_{g_2} f_2 \right\rangle_{L^2(\Z^n \times \T^n)}=\langle f_1, f_2 \rangle_{\ell^2(\Z^n)} \; \langle g_2, g_1 \rangle_{\ell^2(\Z^n)}.
\end{equation}

\item[(2)] $($Plancherel's formula$)$ Let $g \in \ell^2(\Z^n)$ be a non-zero window function. Then for every $f \in \ell^2(\Z^n)$, we have
\begin{equation}\label{eq17}
\left\Vert V_g f \right\Vert_{L^2(\Z^n \times \T^n)} = \Vert f \Vert_{\ell^2(\Z^n)} \; \Vert g \Vert_{\ell^2(\Z^n)}.
\end{equation}

\item[(3)] $($Inversion formula$)$ Let $g, \gamma \in \ell^2(\Z^n)$ and $\langle g, \gamma\rangle_{\ell^2(\Z^n)} \neq 0$. Then for every $f \in \ell^2(\Z^n)$, we have
\[f=\frac{1}{\langle\gamma, g\rangle_{\ell^2(\Z^n)}}
\sum_{m \in \Z^n} \int_{\T^n} V_g f(m,w) \; M_w T_m \gamma  \; dw.\]
\end{itemize}
\end{proposition}
\begin{proof}
(1) Note that we can write $V_g f(m,w)=\mathcal{F}_{\Z^n} (f \cdot T_m\overline{g})(w)$. Using Parseval's formula, we obtain
\begin{eqnarray*}
\left\langle V_{g_1} f_1, V_{g_2} f_2 \right\rangle_{L^2(\Z^n \times \T^n)}
&=& \sum_{m \in \Z^n} \int_{\T^n} V_{g_1} f_1(m,w) \; \overline{V_{g_2} f_2(m,w)} \; dw    \\
&=& \sum_{m \in \Z^n} \int_{\T^n} \mathcal{F}_{\Z^n} (f_1 \cdot T_m\overline{g}_1)(w) \; \overline{\mathcal{F}_{\Z^n} (f_2 \cdot T_m\overline{g}_2)(w)} \; dw  \\
&=& \sum_{m \in \Z^n} \sum_{k \in \Z^n}  (f_1 \cdot T_m\overline{g}_1)(k) \; \overline{(f_2 \cdot T_m\overline{g}_2)(k)} \\
&=& \sum_{m \in \Z^n} \sum_{k \in \Z^n} f_1(k) \;  \overline{g_1(k-m)} \; \overline{f_2(k)}\; g_2(k-m) \\
&=& \sum_{k \in \Z^n} f_1(k) \; \overline{f_2(k)} \sum_{m \in \Z^n} g_2(m)\; \overline{g_1(m)} \\
&=& \langle f_1, f_2 \rangle_{\ell^2(\Z^n)} \; \langle g_2, g_1 \rangle_{\ell^2(\Z^n)}.
\end{eqnarray*}

(2) By taking $g_1=g_2=g$ and $f_1=f_2=f$ in the relation (\ref{eq18}), we obtain the Plancherel formula (\ref{eq17}).

(3) Since $V_g f \in L^2\left(\Z^n \times \T^n  \right)$, the integral 
\[\tilde{f}=\frac{1}{\langle\gamma, g\rangle_{\ell^2(\Z^n)}} \sum_{m \in \Z^n} \int_{\T^n} V_g f(m,w) \; M_w T_m \gamma  \; dw\]
is a well-defined function in $\ell^2(\Z^n)$. Also, for any $h \in \ell^2(\Z^n)$, using the orthogonality relations, we obtain that
\begin{eqnarray*}
\langle \tilde{f}, h \rangle_{\ell^2(\Z^n)}
& = & \frac{1}{\langle\gamma, g\rangle_{\ell^2(\Z^n)}} \sum_{m \in \Z^n} \int_{\T^n} V_g f(m,w) \; \overline{\left\langle h, M_w T_m \gamma \right\rangle}_{\ell^2(\Z^n)} \;dw \\
& = & \frac{1}{\langle\gamma, g\rangle_{\ell^2(\Z^n)}} \left\langle V_g f, V_{\gamma} h \right\rangle_{L^2\left(\Z^n \times \T^n \right)} =\langle f, h\rangle_{\ell^2(\Z^n)}.   
\end{eqnarray*}
Thus $\tilde{f}=f$, and we obtain the inversion formula.
\end{proof}

\section{Modulation spaces on $\Z^n$}\label{sec3}
The modulation spaces were introduced by Feichtinger \cite{fei03, fei97}, by imposing integrability conditions on the STFT of tempered distributions. Here we define the modulation spaces on $\Z^n$ using the STFT on $\Z^n \times \T^n$. 
\begin{Def}
Fix a non-zero window $g \in \mathcal{S}(\Z^n)$, and $1 \leq p \leq \infty$. Then the modulation space $M^{p}(\Z^n)$  consists of all tempered distributions $f \in \mathcal{S'}(\Z^n)$ such that $V_g f \in L^{p}(\Z^n \times \T^n)$. The norm on $M^{p}(\Z^n)$ is 
\begin{eqnarray*}
\Vert f \Vert_{M^{p}(\Z^n)} 
= \Vert V_g f \Vert_{L^{p}(\Z^n \times \T^n)} 
= \bigg( \sum_{m \in \Z^n} \int_{\T^n} |V_g f(m,w)|^p \; dw \bigg)^{1/p} < \infty,
\end{eqnarray*}
with the usual adjustments if $p$ is infinite.  
\end{Def}
The definition of $M^{p}(\Z^n)$ is independent of the choice of $g$ in the sense that each different choice of $g$ defines an equivalent norm on $M^{p}(\Z^n)$. Each modulation space is a Banach space. For $p=2$, we have that $M^2(\Z^n) =\ell^2(\Z^n).$ For other $p$, the space $M^p(\Z^n)$ is not $\ell^p(\Z^n)$. In fact for $p>2$, the space $M^p(\Z^n)$ is a superset of $\ell^2(\Z^n)$. Here we collect some basic properties and inclusion relations of modulation spaces on $\Z^n$. These results can be obtained using simple modifications of the corresponding proofs for the case of modulation spaces defined on $\mathbb R^n$. We define the space of special windows $\mathcal{S}_{\mathcal{C}}(\Z^n)$ by 
\begin{eqnarray*}
&& \mathcal{S}_{\mathcal{C}} \left(\Z^n\right)=\left\{ f \in \ell^{2} (\Z^n) :  \; f=V_{g}^{*} F= \sum_{m \in \Z^n} \int_{\T^n} F(m, w)\; M_{w} T_{m} g  \; dw, \right. \\
&& \left. \qquad \qquad  \qquad \qquad  \qquad  \text{where} \; F \in L^{\infty} (\Z^n \times \T^n) \; \text{and} \; \supp F  \; \text{is compact} \right\}.
\end{eqnarray*}
Then $\mathcal{S}_{\mathcal{C}}\left(\Z^n\right) \subseteq \mathcal{S}\left(\Z^n\right)$ and $\mathcal{S}_{\mathcal{C}}\left(\Z^n \right)$ is dense in $M^1(\Z^n)$. 
Let $B$ be a Banach space of tempered distributions with the following properties: (1) $B$ is invariant under time-frequency shifts, and $\left\|T_{m} M_{w} f\right\|_{B} \leq C  \|f\|_{B}$ for all $f \in B$, (2) $M^1(\Z^n) \cap B \neq\{0\}$. Then $M^1(\Z^n)$ is embedded in $B$ (see \cite{gro01}, Theorem 12.1.9). Also, $M^p(\Z^n)$ is invariant under time-frequency shifts and $\left\|T_{m} M_{w} f\right\|_{M^p(\Z^n)} \leq C  \|f\|_{M^p(\Z^n)}$.  
Since $\mathcal{S}_{\mathcal{C}}\left( \Z^n \right) \subseteq M^1(\Z^n) \cap M^p(\Z^n)$, using a similar method as in Corollary 12.1.10 in \cite{gro01}, we obtain the following inclusions
\[ \mathcal{S}(\Z^n) \subset M^1(\Z^n) \subset M^2(\Z^n)=\ell^2(\Z^n) \subset M^\infty(\Z^n) \subset \mathcal{S'}(\Z^n).\]
In particular, we have $M^p(\Z^n) \hookrightarrow \ell^p(\Z^n)$ for $1 \leq p \leq 2$, and  $\ell^p(\Z^n) \hookrightarrow M^p(\Z^n)$ for $2 \leq p \leq \infty$. Furthermore, the dual of a modulation space is also a modulation space, if $p < \infty$, $(M^{p}(\Z^n))^{'} =M^{p'}(\Z^n)$, where $p'$ is the conjugate exponent of $p$. Next, we present convolution relations which we will use in the proof of the main results. 
\begin{proposition}\label{pro6}
Let $F \in L^1(\Z^n \times \T^n)$ and $G \in L^p(\Z^n \times \T^n)$. Then 
\begin{equation}\label{eq01}
\|F * G\|_{L^p(\Z^n \times \T^n)} \leq \|F\|_{L^1(\Z^n \times \T^n)} \; \|G\|_{L^p(\Z^n \times \T^n)}.
\end{equation}
\end{proposition}
\begin{proof}
Let $H \in  L^{p^{\prime}}(\Z^n \times \T^n)$. Using H\"older's inequality, we get 
\begin{eqnarray*}
|\langle F * G, H \rangle| 
& = & \left| \sum_{m \in \Z^n}  \int_{\T^n} F*G(m,w) \; \overline{H(m,w)} \; dw \right| \\
& \leq & \sum_{m \in \Z^n}  \int_{\T^n} \left( \sum_{l \in \Z^n}  \int_{\T^n} \left| G(m-l, w-x) \right|\; |F(l,x)| \; dx \right) |H(m,w)| \; dw \\
& = &  \sum_{l \in \Z^n}  \int_{\T^n} \left( \sum_{m \in \Z^n}  \int_{\T^n} \left| T_{(l,x)} G(m, w) \right|\; |H(m,w)| \; dw \right) |F(l,x)| \; dx  \\
& \leq & \sum_{l \in \Z^n}  \int_{\T^n}  |F(l,x)| \; \| T_{(l,x)} G \|_{L^p(\Z^n \times \T^n)} \; \|H\|_{L^{p^{\prime}}(\Z^n \times \T^n)} \; dx \\
& = &  \sum_{l \in \Z^n}  \int_{\T^n}  |F(l,x)| \; dx \; \| G \|_{L^p(\Z^n \times \T^n)} \; \|H\|_{L^{p^{\prime}}(\Z^n \times \T^n)} \\
& = & \|F\|_{L^1(\Z^n \times \T^n)} \; \| G \|_{L^p(\Z^n \times \T^n)} \; \|H\|_{L^{p^{\prime}}(\Z^n \times \T^n)}.
\end{eqnarray*}
By duality, we have
\[ \|F * G\|_{L^p(\Z^n \times \T^n)} = \sup\left\{|\langle F * G, H \rangle|:\|H\|_{L^{p^{\prime}}(\Z^n \times \T^n)}  \leq 1 \right\} 
\leq \|F\|_{L^1(\Z^n \times \T^n)} \; \|G\|_{L^p(\Z^n \times \T^n)}. \]
\end{proof}

Next, we define the modulation spaces on $\Z^n \times \T^n$ using the STFT on $(\Z^{n} \times \T^{n}) \times (\Z^{n} \times \T^{n})^{\wedge}$. For $g \in \mathcal{S}\left(\mathbb{Z}^n \times \mathbb{T}^n\right) \backslash\{0\}$ and $1 \leq p \leq \infty$, the modulation space $M^{p}(\Z^n \times \T^n)$ consists of all tempered distributions $F \in \mathcal{S'}\left(\mathbb{Z}^n \times \mathbb{T}^n\right) $  such that $V_g F \in L^p((\Z^{n} \times \T^{n}) \times (\Z^{n} \times \T^{n})^{\wedge})$. The norm on $M^{p}(\Z^n \times \T^n)$ is defined by
$\Vert F \Vert_{M^{p}(\Z^n \times \T^n)}  = \Vert V_g F \Vert_{L^p((\Z^{n} \times \T^{n}) \times (\Z^{n} \times \T^{n})^{\wedge})} $. Let $m, k \in \mathbb{Z}^n$ and $\omega, \xi \in \mathbb{T}^n$. The STFT $V_g F$ on $(\Z^{n} \times \T^{n}) \times (\Z^{n} \times \T^{n})^{\wedge}$ is defined as 
\[ V_g F (m, \omega, \xi, k)=\sum_{j \in \mathbb{Z}^n} \int_{\mathbb{T}^n} e^{-2 \pi i j \xi} e^{-2 \pi i \eta k} F(j, \eta) \overline{g(j-m, \eta-\omega)} d \eta. \]
Here $(m, \omega) \in \mathbb{Z}^n \times \mathbb{T}^n$ and $(\xi, k) \in \mathbb{T}^n \times \mathbb{Z}^n=\left(\mathbb{Z}^n \times \mathbb{T}^n\right)^{\wedge}$. For $j \in \mathbb{Z}^n, \eta \in \mathbb{T}^n$, we define $\tilde{g}(j, \eta) =g(-j,-\eta)$. Then, we can write the STFT $V_g F$ as 
\[ \begin{aligned}
V_g F(m, \omega, \xi, k) & =\sum_{j \in \mathbb{Z}^n} \int_{\mathbb{T}^n} e^{-2 \pi i m \xi} e^{-2 \pi i \omega k} F(j, \eta) e^{2 \pi i(m-j) \xi} e^{2 \pi i(\omega-\eta) k} \overline{\tilde{g}(m-j, \omega-\eta)} d \eta \\
& =e^{-2 \pi i m \xi} e^{-2 \pi i \omega k}\left(F * M_{(\xi, k)} \overline{\tilde{g}}\right)(m, \omega),
\end{aligned} \]
where $M_{(\xi, k)}$ is the modulation operator, that is given by
\[ \left(M_{(\xi, k)} h\right)(j, \eta)=e^{2 \pi i j \xi} e^{2 \pi i k \eta} h(j, \eta), \quad j, k \in \mathbb{Z}^n, \;\;  \xi, \eta \in \mathbb{T}^n. \] 
Similarly, the inclusion relations of modulation spaces on $\Z^n \times \T^n$ can be obtained using the same techniques as in modulation spaces on locally compact abelian groups (see \cite{bas22, fei03}). We refer to Gr\"ochenig's book \cite{gro01} for further properties and uses of modulation spaces.

\section{Localization operators on discrete modulation spaces}\label{sec4}

In this section, we define the localization operators on $\Z^n$ and we show that these operators are bounded. Also, we prove that localization operators are compact and in the Schatten--von Neumann class. 

\begin{Def}
Let $\varsigma \in L^1(\Z^n \times \T^n)  \cup  L^\infty(\Z^n \times \T^n)$. The localization operator associated with the symbol $\varsigma$ and two window functions $g_1$ and $g_2$ in $\mathcal{S} (\mathbb Z^n)$, is denoted by  $\mathfrak{L}^{g_1, g_2}_{\varsigma}$, and defined on $\ell^2(\Z^n)$, by
\begin{equation}\label{eq41}
\mathfrak{L}^{g_1, g_2}_{\varsigma}f(k)=\sum_{m \in \Z^n}  \int_{\T^n} \varsigma(m, w) \; V_{g_1}f(m,w) \; M_w T_m g_2(k) \; dw, \quad k \in \Z^n. 
\end{equation}  
Also, it is useful to rewrite the definition of $\mathfrak{L}^{g_1, g_2}_{\varsigma}$ in a weak sense as, for every $f,  h \in \ell^2(\Z^n)$ 
\begin{equation}\label{eq42}
\left\langle \mathfrak{L}^{g_1, g_2}_{\varsigma}f , h \right\rangle_{\ell^2(\Z^n)} = \sum_{m \in \Z^n}  \int_{\T^n} \varsigma(m, w) \; V_{g_1}f(m,w) \; \overline{V_{g_2}h(m,w)} \; dw.
\end{equation}
\end{Def}

We denote by $\mathcal{B}(\ell^p(\Z^n))$, $1 \leq p \leq \infty$, the space of all bounded linear operators from $\ell^p(\Z^n)$ into itself. In particular, $\mathcal{B}(\ell^2(\Z^n))$ denotes the C$^*$-algebra of bounded linear operator $\mathcal{A}$ from $\ell^2(\Z^n)$ into itself, equipped with the norm 
\[ \Vert \mathcal{A} \Vert_{\mathcal{B}(\ell^2(\Z^n))}= \sup_{\Vert f \Vert_{\ell^2(\Z^n)} \leq 1} \Vert \mathcal{A}(f) \Vert_{\ell^2(\Z^n)}. \]

Next, we define the Schatten--von Neumann class $S_p$ on $\Z^n$. For a compact operator $\mathcal{A} \in \mathcal{B}(\ell^2(\Z^n))$, the eigenvalues of the positive self-adjoint operator $|\mathcal{A}|=\sqrt{\mathcal{A}^* \mathcal{A}}$ are called the singular values of $\mathcal{A}$ and denoted by $\{ s_n(\mathcal{A}) \}_{n \in \mathbb{N}}$. For $1 \leq p < \infty$, the Schatten--von Neumann class $S_p$ is defined to be the space of all compact operators whose singular values lie in $\ell^p$. $S_p$ is equipped with the norm 
\[ \Vert \mathcal{A} \Vert_{S_p}= \left(\sum_{n=1}^\infty  (s_n(\mathcal{A}))^p \right)^{1/p}. \]
For $p=\infty$,  the Schatten--von Neumann class $S_\infty$ is the class of all compact operators with the norm  $\Vert \mathcal{A} \Vert_{S_\infty} :=\Vert \mathcal{A} \Vert_{\mathcal{B}(\ell^2(\Z^n))}$. In particular, for $p = 1$, we define the trace of an operator $\mathcal{A}$ in $S_1$ by 
\[ tr(\mathcal{A}) = \sum_{n=1}^\infty \langle \mathcal{A}  v_n, v_n  \rangle_{\ell^2(\Z^n)}, \]
where $\{v_n \}_n$ is any orthonormal basis of $\ell^2(\Z^n)$. Moreover, if $\mathcal{A}$ is positive, then 
\[ tr(\mathcal{A})=\Vert \mathcal{A} \Vert_{S_1} .\]
A compact operator $\mathcal{A}$ on the Hilbert space $\ell^2(\Z^n)$ is called Hilbert--Schmidt, if the positive operator $\mathcal{A}^*  \mathcal{A}$ is in the trace class $S_1$. Then for any orthonormal basis $\{v_n \}_n$ of $\ell^2(\Z^n)$, we have 
\[ \Vert \mathcal{A}\Vert_{HS}^2 :=\Vert \mathcal{A}\Vert_{S_2}^2= \Vert \mathcal{A}^* \mathcal{A} \Vert_{S_1}= tr(\mathcal{A}^* \mathcal{A}) = \sum_{n=1}^\infty \Vert \mathcal{A} v_n\Vert^2_{\ell^2(\Z^n)}. \]

\subsection{Boundedness and compactness of localization operators}

In this subsection, we consider window functions $g_1, g_2 \in  M^1(\Z^n)$ and establish the following boundedness and compactness results of localization operators. 

\begin{proposition}\label{pro01}
Let $\varsigma \in L^\infty(\Z^n \times \T^n)$ and $g_1, g_2 \in  M^1(\Z^n)$. Then the localization operator $\mathfrak{L}^{g_1, g_2}_{\varsigma}$ is in $\mathcal{B}(\ell^2(\Z^n))$ and we have
\[ \left\Vert \mathfrak{L}^{g_1, g_2}_{\varsigma} \right\Vert_{\mathcal{B}(\ell^2(\Z^n))} \leq \Vert \varsigma \Vert_{L^\infty(\Z^n \times \T^n)} \; \Vert g_1 \Vert_{M^1(\Z^n)} \; \Vert g_2 \Vert_{M^1(\Z^n)}. \]
\end{proposition}
\begin{proof}
For every $f, h \in \ell^2(\Z^n)$, using
H\"older's inequality,  we obtain
\begin{eqnarray*}
\left| \left\langle \mathfrak{L}^{g_1, g_2}_{\varsigma}f, h \right\rangle_{\ell^2(\Z^n)} \right| 
& \leq & \sum_{m \in \Z^n}  \int_{\T^n} |\varsigma(m, w)| \; \left|V_{g_1}f(m,w) \right|\; \left|V_{g_2}h(m,w) \right| \; dw \\ 
& \leq & \left\Vert \varsigma \right\Vert_{L^\infty(\Z^n \times \T^n)} \left\Vert V_{g_1}f  \right\Vert_{L^2(\Z^n \times \T^n)} \left\Vert V_{g_2}h \right\Vert_{L^2(\Z^n \times \T^n)}.
\end{eqnarray*}
Using Plancherel's formula (\ref{eq17}), we get 
\begin{eqnarray*}
\left| \left\langle \mathfrak{L}^{g_1, g_2}_{\varsigma}f, h \right\rangle_{\ell^2(\Z^n)} \right| 
\leq  \left\Vert \varsigma \right\Vert_{L^\infty(\Z^n \times \T^n)} \Vert f \Vert_{\ell^2(\Z^n)} \; \Vert g_1 \Vert_{\ell^2(\Z^n)} \; \Vert h  \Vert_{\ell^2(\Z^n)} \; \Vert g_2 \Vert_{\ell^2(\Z^n)}.
\end{eqnarray*}
Since $M^1(\Z^n) \subset \ell^2(\Z^n)$, we have 
\[\Vert g_1 \Vert_{\ell^2(\Z^n)} \leq \Vert g_1 \Vert_{M^1(\Z^n)} \quad \text{and} \quad  \Vert g_2  \Vert_{\ell^2(\Z^n)} \leq \Vert g_2 \Vert_{M^1(\Z^n)}. \]
Hence,
\[ \left\Vert \mathfrak{L}^{g_1, g_2}_{\varsigma} \right\Vert_{\mathcal{B}(\ell^2(\Z^n))} \leq \Vert \varsigma \Vert_{L^\infty(\Z^n \times \T^n)}\; \Vert g_1 \Vert_{M^1(\Z^n)} \; \Vert g_2 \Vert_{M^1(\Z^n)}. \]
\end{proof}

\begin{proposition}\label{pro1}
Let $\varsigma \in M^1(\Z^n \times \T^n)$ and $g_1, g_2 \in  M^1(\Z^n)$. Then the localization operator $\mathfrak{L}^{g_1, g_2}_{\varsigma}$ is in $\mathcal{B}(\ell^2(\Z^n))$ and we have 
$$\left\Vert \mathfrak{L}^{g_1, g_2}_{\varsigma} \right\Vert_{\mathcal{B}(\ell^2(\Z^n))} \leq \Vert \varsigma \Vert_{M^1(\Z^n \times \T^n)} \; \Vert g_1 \Vert_{M^1(\Z^n)} \; \Vert g_2 \Vert_{M^1(\Z^n)}. $$
\end{proposition}
\begin{proof}
Let $f , h \in \ell^2(\Z^n)$. Since $M^\infty(\Z^n \times \T^n)$ is the dual space of $M^1(\Z^n \times \T^n)$, we have 
\begin{eqnarray}\label{eq45}
\left| \left\langle \mathfrak{L}^{g_1, g_2}_{\varsigma}f, h \right\rangle_{\ell^2(\Z^n)} \right|   
& \leq & \sum_{m \in \Z^n}  \int_{\T^n} |\varsigma(m, w)| \; \left|V_{g_1}f(m, w) \; \overline{V_{g_2}h(m, w)} \right|  \; dw \nonumber \\
& \leq & \Vert \varsigma \Vert_{M^1(\Z^n \times \T^n)} \left\Vert V_{g_1}f \cdot \overline{V_{g_2}h} \right\Vert_{M^\infty(\Z^n \times \T^n)}.
\end{eqnarray}
Since $L^2(\Z^n \times \T^n) \subset M^\infty(\Z^n \times \T^n)$ and $M^1(\Z^n) \subset \ell^2(\Z^n)$, using Plancherel's formula (\ref{eq17}), we obtain 
\begin{eqnarray}\label{eq46}
\left\Vert V_{g_1}f \cdot \overline{V_{g_2}h} \right\Vert_{M^\infty(\Z^n \times \T^n)}
& \leq & \left\Vert V_{g_1}f \cdot \overline{V_{g_2}h} \right\Vert_{L^2(\Z^n \times \T^n)} \nonumber \\
& \leq & \left\Vert V_{g_1}f \right\Vert_{L^2(\Z^n \times \T^n)} \left\Vert V_{g_2}h \right\Vert_{L^2(\Z^n \times \T^n)} \nonumber \\
& = & \Vert f \Vert_{\ell^2(\Z^n)} \; \Vert g_1 \Vert_{\ell^2(\Z^n)} \; \Vert h \Vert_{\ell^2(\Z^n)} \; \Vert g_2 \Vert_{\ell^2(\Z^n)} \nonumber \\
& \leq & \Vert f \Vert_{\ell^2(\Z^n)} \; \Vert h \Vert_{\ell^2(\Z^n)} \; \Vert g_1 \Vert_{M^1(\Z^n)} \; \Vert g_2 \Vert_{M^1(\Z^n)}.
\end{eqnarray}
Thus from (\ref{eq45}) and (\ref{eq46}), we get  
\[\left\Vert \mathfrak{L}^{g_1, g_2}_{\varsigma} \right\Vert_{\mathcal{B}(\ell^2(\Z^n))} \leq \Vert \varsigma \Vert_{M^1(\Z^n \times \T^n)} \;  \Vert g_1 \Vert_{M^1(\Z^n)} \; \Vert g_2 \Vert_{M^1(\Z^n)} .\]
\end{proof}

\begin{proposition}\label{pro2}
Let $\varsigma \in M^2(\Z^n \times \T^n)$ and $g_1, g_2 \in  M^1(\Z^n)$. Then the localization operator $\mathfrak{L}^{g_1, g_2}_{\varsigma}$ is in $\mathcal{B}(\ell^2(\Z^n))$ and we have
\[ \left\Vert \mathfrak{L}^{g_1, g_2}_{\varsigma} \right\Vert_{\mathcal{B}(\ell^2(\Z^n))} \leq \Vert \varsigma \Vert_{M^2(\Z^n \times \T^n)} \; \Vert g_1 \Vert_{M^1(\Z^n)} \; \Vert g_2 \Vert_{M^1(\Z^n)}. \]
\end{proposition}    
\begin{proof}
For every $f, h \in \ell^2(\Z^n)$, using H\"older's inequality, we deduce that
\begin{eqnarray*}
\left| \left\langle \mathfrak{L}^{g_1, g_2}_{\varsigma}f, h \right\rangle_{\ell^2(\Z^n)} \right| 
& \leq & \sum_{m \in \Z^n}  \int_{\T^n} |\varsigma(m, w)| \; \left|V_{g_1}f(m, w) \; \overline{V_{g_2}h(m, w)} \right| \; dw \\ 
& \leq & \left\Vert \varsigma \right\Vert_{L^2(\Z^n \times \T^n)} \left\Vert V_{g_1}f  \cdot \overline{V_{g_2}h} \right\Vert_{L^2(\Z^n \times \T^n)}.
\end{eqnarray*}
Since $L^2(\Z^n \times \T^n)= M^2(\Z^n \times \T^n)$, using the estimate obtained in (\ref{eq46}), we get
\begin{eqnarray*}
\left| \left\langle \mathfrak{L}^{g_1, g_2}_{\varsigma}f, h \right\rangle_{\ell^2(\Z^n)} \right|
\leq \left\Vert \varsigma \right\Vert_{M^2(\Z^n \times \T^n)} \Vert f \Vert_{\ell^2(\Z^n)} \; \Vert h \Vert_{\ell^2(\Z^n)} \; \Vert g_1 \Vert_{M^1(\Z^n)} \; \Vert g_2 \Vert_{M^1(\Z^n)}.
\end{eqnarray*}
Hence,
\[ \left\Vert \mathfrak{L}^{g_1, g_2}_{\varsigma} \right\Vert_{\mathcal{B}(\ell^2(\Z^n))} \leq \Vert \varsigma \Vert_{M^2(\Z^n \times \T^n)} \; \Vert g_1 \Vert_{M^1(\Z^n)} \; \Vert g_2 \Vert_{M^1(\Z^n)}. \]
\end{proof}

\begin{theorem}\label{th1}
Let $\varsigma \in M^p(\Z^n \times \T^n)$, $1 < p < 2$ and $g_1, g_2 \in  M^1(\Z^n)$. Then, for fixed $\varsigma \in M^p(\Z^n \times \T^n)$ the operator $\mathfrak{L}_{g_1, g_2}$ can be uniquely extended to a bounded linear operator on $\ell^2(\Z^n)$, such that
\[ \left\Vert \mathfrak{L}^{g_1, g_2}_{\varsigma} \right\Vert_{\mathcal{B}(\ell^2(\Z^n))} \leq \Vert \varsigma \Vert_{M^p(\Z^n \times \T^n)} \; \Vert g_1 \Vert_{M^1(\Z^n)} \; \Vert g_2 \Vert_{M^1(\Z^n)}. \] 
\end{theorem}
\begin{proof}
Let $1< p < 2$. For every $\varsigma \in M^1(\Z^n \times \T^n) \cap  M^2(\Z^n \times \T^n)$, using Proposition \ref{pro1}, Proposition \ref{pro2}, the fact that the modulation spaces $M^{p}$ interpolate exactly like the corresponding mixed-norm spaces $L^{p}$ and the Riesz--Thorin interpolation theorem (see \cite{ste56}), we obtain 
\[ \left\Vert \mathfrak{L}^{g_1, g_2}_{\varsigma} \right\Vert_{\mathcal{B}(\ell^2(\Z^n))} \leq \Vert \varsigma \Vert_{M^p(\Z^n \times \T^n)} \; \Vert g_1 \Vert_{M^1(\Z^n)} \; \Vert g_2 \Vert_{M^1(\Z^n)} . \]
Let $\varsigma \in M^p(\Z^n \times \T^n)$ and $\{ \varsigma_n \}_{n \geq 1}$ be a sequence of functions in $M^1(\Z^n \times \T^n) \cap M^2(\Z^n \times \T^n)$ such that $\varsigma_n \to \varsigma$ in $M^p(\Z^n \times \T^n)$ as $ n \to \infty$. Hence for every $n, k \in \mathbb{N}$, we have 
\[ \left\Vert \mathfrak{L}^{g_1, g_2}_{\varsigma_n} - \mathfrak{L}^{g_1, g_2}_{\varsigma_k} \right\Vert_{\mathcal{B}(\ell^2(\Z^n))} \leq \Vert \varsigma_n - \varsigma_k \Vert_{M^p(\Z^n \times \T^n)} \; \Vert g_1 \Vert_{M^1(\Z^n)} \; \Vert g_2 \Vert_{M^1(\Z^n)} . \]
Therefore, $\{ \mathfrak{L}^{g_1, g_2}_{\varsigma_n} \}_{n \geq 1}$ is a Cauchy sequence in $\mathcal{B}(\ell^2(\Z^n))$. Let $\{ \mathfrak{L}^{g_1, g_2}_{\varsigma_n} \}_{n \geq 1}$ converges to $ \mathfrak{L}^{g_1, g_2}_{\varsigma}$. Then the limit $ \mathfrak{L}^{g_1, g_2}_{\varsigma}$ is independent of the choice of $\{ \varsigma_n \}_{n \geq 1}$ and we obtain
\begin{eqnarray*}
\left\Vert \mathfrak{L}^{g_1, g_2}_{\varsigma} \right\Vert_{\mathcal{B}(\ell^2(\Z^n))} 
=  \lim_{n \to \infty} \left\Vert \mathfrak{L}^{g_1, g_2}_{\varsigma_n} \right\Vert_{\mathcal{B}(\ell^2(\Z^n))} 
& \leq & \lim_{n \to \infty} \Vert \varsigma_n \Vert_{M^p(\Z^n \times \T^n)} \; \Vert g_1 \Vert_{M^1(\Z^n)} \; \Vert g_2 \Vert_{M^1(\Z^n)} \\
& =& \Vert \varsigma \Vert_{M^p(\Z^n \times \T^n)} \; \Vert g_1 \Vert_{M^1(\Z^n)} \; \Vert g_2 \Vert_{M^1(\Z^n)}.
\end{eqnarray*}
This completes the proof. 
\end{proof}

\begin{theorem}\label{th2}
Let $\varsigma \in M^p(\Z^n \times \T^n)$, $1 \leq p \leq 2$ and $g_1, g_2 \in  M^1(\Z^n)$. Then the localization operator $\mathfrak{L}^{g_1, g_2}_{\varsigma}: \ell^2(\Z^n) \to \ell^2(\Z^n) $ is compact. 
\end{theorem}
\begin{proof}
Assume that $\varsigma \in M^1(\Z^n \times \T^n)$. Let $\{ v_n \}_n$ be an orthonormal basis for $\ell^2(\Z^n)$. Since $M^1(\Z^n \times \T^n) \subset L^1(\Z^n \times \T^n)$, using Parseval's identity, we obtain 
\begin{eqnarray*}
&& \sum_{n=1}^\infty \left\langle \mathfrak{L}^{g_1, g_2}_{\varsigma} v_n, v_n \right\rangle_{\ell^2(\Z^n)} \\
&& = \sum_{n=1}^\infty  \sum_{m \in \Z^n}  \int_{\T^n} \varsigma (m,w) \; \langle v_n, M_w T_m g_1 \rangle_{\ell^2(\Z^n)} \; \langle M_w T_m g_2 , v_n \rangle_{\ell^2(\Z^n)} \; dw\\
&& = \sum_{m \in \Z^n}  \int_{\T^n} \varsigma (m,w)  \left( \sum_{n=1}^\infty \langle v_n, M_w T_m g_1 \rangle_{\ell^2(\Z^n)}\;  \langle M_w T_m g_2 , v_n \rangle_{\ell^2(\Z^n)} \right) dw \\
&& \leq \frac{1}{2} \sum_{m \in \Z^n}  \int_{\T^n} \varsigma (m,w)  \left( \sum_{n=1}^\infty |\langle v_n, M_w T_m g_1 \rangle_{\ell^2(\Z^n)}|^2 + \sum_{n=1}^\infty |\langle M_w T_m g_2, v_n \rangle_{\ell^2(\Z^n)}|^2 \right) \; dw\\
&& = \frac{1}{2} \Vert \varsigma \Vert_{L^1(\Z^n \times \T^n)} \; (\Vert g_1 \Vert^2_{\ell^2(\Z^n)} + \Vert g_2 \Vert^2_{\ell^2(\Z^n)}) \\
&& \leq \frac{1}{2} \Vert \varsigma \Vert_{M^1(\Z^n \times \T^n)} \; (\Vert g_1 \Vert^2_{M^1(\Z^n)} + \Vert g_2 \Vert^2_{M^1(\Z^n)}). \qquad \qquad \qquad \qquad \qquad \qquad \qquad \qquad \qquad \qquad \qquad
\end{eqnarray*}
Therefore, the operator $\mathfrak{L}^{g_1, g_2}_{\varsigma}$ is in $S_1$. Next, assume that $\varsigma \in M^p(\Z^n \times \T^n)$. We consider a sequence of functions $\{ \varsigma_n \}_{n \geq 1}$ in $M^1(\Z^n \times \T^n) \cap M^2(\Z^n \times \T^n)$ such that $\varsigma_n \to \varsigma$ in $M^p(\Z^n \times \T^n)$ as $n \to \infty$. Then, using Theorem \ref{th1}, we get
\[ \Vert \mathfrak{L}^{g_1, g_2}_{\varsigma_n} - \mathfrak{L}^{g_1, g_2}_{\varsigma}  \Vert_{\mathcal{B}(\ell^2(\Z^n))} \leq \Vert \varsigma_n - \varsigma \Vert_{M^p(\Z^n \times \T^n)} \; \Vert g_1 \Vert_{M^1(\Z^n)} \; \Vert g_2 \Vert_{M^1(\Z^n)}  \to 0, \]
as $n \to \infty$. Hence, $\mathfrak{L}^{g_1, g_2}_{\varsigma_n} \to \mathfrak{L}^{g_1, g_2}_{\varsigma} $ in $\mathcal{B}(\ell^2(\Z^n))$ as $n \to \infty$. From the above, we obtain that $ \{ \mathfrak{L}^{g_1, g_2}_{\varsigma_n} \}_{n \geq 1}$ is a sequence of linear operators in $S_1$ and hence compact, so $\mathfrak{L}^{g_1, g_2}_{\varsigma}$ is compact. 
\end{proof}

Next, we calculate the adjoint $(\mathfrak{L}^{g_1, g_2}_{\varsigma})^*$ of the operator $\mathfrak{L}^{g_1, g_2}_{\varsigma}$ on $\ell^2(\Z^n)$ determined by the relation
\[ \left\langle \mathfrak{L}^{g_1, g_2}_{\varsigma}f , h \right\rangle_{\ell^2(\Z^n)}= \left\langle f , (\mathfrak{L}^{g_1, g_2}_{\varsigma})^*h \right\rangle_{\ell^2(\Z^n)}. \] 
We have
\begin{eqnarray*}
\left\langle \mathfrak{L}^{g_1, g_2}_{\varsigma}f , h \right\rangle_{\ell^2(\Z^n)} 
&=& \sum_{m \in \Z^n}  \int_{\T^n} \varsigma(m, w) \; V_{g_1}f(m,w) \; \overline{V_{g_2}h(m,w)} \; dw \\
&=& \sum_{m \in \Z^n}  \int_{\T^n} \varsigma(m, w) \; \langle f, M_w T_m g_1 \rangle_{\ell^2(\Z^n)}  \; \overline{V_{g_2}h(m,w)} \; dw \\
&=& \sum_{m \in \Z^n}  \int_{\T^n} \varsigma(m, w) \; \langle f, V_{g_2}h(m,w) \; M_w T_m g_1 \rangle_{\ell^2(\Z^n)}  \; dw \\
&=&  \left\langle f, \sum_{m \in \Z^n}  \int_{\T^n} \overline{\varsigma(m, w)} \; V_{g_2}h(m,w) \; M_w T_m g_1 \; dw \right\rangle_{\ell^2(\Z^n)} \\
&=& \left\langle f, \mathfrak{L}^{g_2, g_1}_{\overline{\varsigma}}h \right\rangle_{\ell^2(\Z^n)}.
\end{eqnarray*}
Therefore, we have
\[ (\mathfrak{L}^{g_1, g_2}_{\varsigma})^* = \mathfrak{L}^{g_2, g_1}_{\overline{\varsigma}}. \]
Hence, the localization operator $\mathfrak{L}^{g_1, g_2}_{\varsigma}$ is a self-adjoint operator if $g_1=g_2$ and the symbol $\varsigma$ is a real-valued function.  

\subsection{Localization operators in $S_p$}

In this subsection, we prove that the localization operator $\mathfrak{L}^{g, g}_{\varsigma}$ is in $S_p$ and provide an upper bound of the norm $\Vert \mathfrak{L}^{g, g}_{\varsigma} \Vert_{S_p}$. We begin with the following proposition.

\begin{proposition}\label{pro3}
Let $\varsigma \in  M^1(\Z^n \times \T^n)$ and $g \in  M^1(\Z^n)$. Then the localization operator $\mathfrak{L}^{g, g}_{\varsigma}: \ell^2(\Z^n) \to \ell^2(\Z^n) $ is in $S_1$ and 
\[ \Vert \mathfrak{L}^{g, g}_{\varsigma} \Vert_{S_1} \leq 4 \Vert \varsigma \Vert_{M^1(\Z^n \times \T^n)} \;  \Vert g \Vert^2_{M^1(\Z^n)}. \]
\end{proposition}
\begin{proof}
If $\varsigma \in  M^1(\Z^n \times \T^n)$, then from the first part of the proof of Theorem \ref{th2}, the operator $\mathfrak{L}^{g, g}_{\varsigma}$ is in $S_1$. Now, to prove the estimate, assume that $\varsigma$ is non-negative real-valued and in $M^1(\Z^n \times \T^n)$. Then $(({\mathfrak{L}^{g, g}_{\varsigma}})^* \mathfrak{L}^{g, g}_{\varsigma})^{1/2}=\mathfrak{L}^{g, g}_{\varsigma}$. Let $\{v_n\}_n$ be an orthonormal basis for $\ell^2(\Z^n)$ consisting of eigenvalues of $(({\mathfrak{L}^{g, g}_{\varsigma}})^* \mathfrak{L}^{g, g}_{\varsigma})^{1/2}: \ell^2(\Z^n) \to \ell^2(\Z^n)$. 
Then by using the estimate obtained in the first part of the proof of Theorem \ref{th2}, we get
\begin{eqnarray}\label{eq43}
\Vert \mathfrak{L}^{g, g}_{\varsigma} \Vert_{S_1}
& =& \sum_{n=1}^\infty \left\langle (({\mathfrak{L}^{g, g}_{\varsigma}})^* \mathfrak{L}^{g, g}_{\varsigma})^{1/2}v_n, v_n \right\rangle_{\ell^2(\Z^n)} \nonumber \\
&=& \sum_{n=1}^\infty \left\langle \mathfrak{L}^{g, g}_{\varsigma} v_n, v_n \right\rangle_{\ell^2(\Z^n)}  
\leq  \Vert \varsigma \Vert_{M^1(\Z^n \times \T^n)} \; \Vert g \Vert^2_{M^1(\Z^n)} .
\end{eqnarray}
Next, assume that $\varsigma \in M^1(\Z^n \times \T^n)$ is  an arbitrary real-valued function. We can write $\varsigma = \varsigma_+ - \varsigma_-$, where $\varsigma_+=\max(\varsigma, 0)$ and $\varsigma_-=-\min(\varsigma, 0)$. Then, using relation (\ref{eq43}), we obtain
\begin{eqnarray}\label{eq44}
\Vert \mathfrak{L}^{g, g}_{\varsigma} \Vert_{S_1}
= \Vert \mathfrak{L}^{g, g}_{\varsigma_+} -\mathfrak{L}^{g, g}_{\varsigma_-} \Vert_{S_1} 
& \leq & \Vert \mathfrak{L}^{g, g}_{\varsigma_+} \Vert_{S_1} +\Vert \mathfrak{L}^{g, g}_{\varsigma_-} \Vert_{S_1} \nonumber \\
& \leq &  \Vert g \Vert^2_{M^1(\Z^n)} (\Vert \varsigma_+ \Vert_{M^1(\Z^n \times \T^n)} + \Vert \varsigma_- \Vert_{M^1(\Z^n \times \T^n)}) \nonumber \\
& \leq & 2  \Vert g \Vert^2_{M^1(\Z^n)} \Vert \varsigma \Vert_{M^1(\Z^n \times \T^n)}.
\end{eqnarray}
Finally, assume that $\varsigma \in  M^1(\Z^n \times \T^n)$ is a complex-valued function. Then, we can write $\varsigma=\varsigma_1+ i \varsigma_2$, where $\varsigma_1, \varsigma_2$ are the real and imaginary parts of $\varsigma$ respectively. Then, using relation (\ref{eq44}), we obtain 
\begin{eqnarray*}
\Vert \mathfrak{L}^{g, g}_{\varsigma} \Vert_{S_1}
= \Vert \mathfrak{L}^{g, g}_{\varsigma_1}+i \; \mathfrak{L}^{g, g}_{\varsigma_2} \Vert_{S_1} 
& \leq & \Vert \mathfrak{L}^{g, g}_{\varsigma_1} \Vert_{S_1} + \Vert \mathfrak{L}^{g, g}_{\varsigma_2} \Vert_{S_1} \\
&\leq & 2  \Vert g \Vert^2_{M^1(\Z^n)} (\Vert \varsigma_1 \Vert_{M^1(\Z^n \times \T^n)} + \Vert \varsigma_2 \Vert_{M^1(\Z^n \times \T^n)}) \nonumber \\
& \leq & 4 \Vert \varsigma \Vert_{M^1(\Z^n \times \T^n)}  \Vert g \Vert^2_{M^1(\Z^n)}  .
\end{eqnarray*}
This completes the proof.
\end{proof}

\begin{theorem}
Let $\varsigma \in  M^1(\Z^n \times \T^n)$ and $g \in  M^1(\Z^n)$. Then the localization operator $\mathfrak{L}^{g, g}_{\varsigma}: \ell^2(\Z^n) \to \ell^2(\Z^n)$ is in $S_p$, for $1 \leq p \leq \infty$ and 
\[\Vert \mathfrak{L}^{g, g}_{\varsigma} \Vert_{S_p} \leq 2^{2/p} \Vert \varsigma  \Vert_{M^1(\Z^n \times \T^n)} \;  \Vert g \Vert^{2}_{M^1(\Z^n)}.\]
\end{theorem}
\begin{proof}
The proof follows from Proposition \ref{pro2}, Proposition \ref{pro3} and by interpolation theorems (see \cite{won02}, Theorems 2.10 and 2.11). 
\end{proof}

Next, we improve the constant given in the previous proposition and also we provide a lower bound of the norm $\Vert \mathfrak{L}^{g, g}_{\varsigma} \Vert_{S_1}$, more precisely we have the following.

\begin{theorem}\label{th4}
Let $\varsigma \in M^1(\Z^n \times \T^n)$ and $g \in  M^1(\Z^n)$. Then the localization operator $\mathfrak{L}^{g, g}_{\varsigma}$ is in $S_1$ and we have
\begin{eqnarray*}
\frac{1}{\Vert g \Vert^2_{M^1(\Z^n)}}\; \Vert \tilde{\varsigma} \Vert_{L^1(\Z^n \times \T^n)}
\leq \Vert \mathfrak{L}^{g, g}_{\varsigma} \Vert_{S_1} \leq \Vert g \Vert^2_{M^1(\Z^n)} \; \Vert \varsigma \Vert_{M^1(\Z^n \times \T^n)},
\end{eqnarray*}
where $\tilde{\varsigma}$ is given by $\tilde{\varsigma}(m, w)=\left\langle \mathfrak{L}^{g, g}_{\varsigma}(M_w T_m g), M_w T_m g \right\rangle_{\ell^2(\Z^n)}$.
\end{theorem}
\begin{proof}
Since $\varsigma \in M^1(\Z^n \times \T^n)$, by Proposition \ref{pro3}, $\mathfrak{L}^{g, g}_{\varsigma}$ is in $S_1$. Using the canonical form of compact operators (see \cite{won02}, Theorem 2.2), we obtain 
\begin{equation}\label{eq49}
\mathfrak{L}^{g, g}_{\varsigma}f =\sum_{n=1}^{\infty} s_n(\mathfrak{L}^{g, g}_{\varsigma}) \langle f, v_n \rangle_{\ell^2(\Z^n)} u_n,
\end{equation}
where $\{s_n(\mathfrak{L}^{g, g}_{\varsigma}) \}_n$ are the positive singular values of $\mathfrak{L}^{g, g}_{\varsigma}$, $\{ v_n \}_n$ is an orthonormal basis for the orthogonal complement of the null space of $\mathfrak{L}^{g, g}_{\varsigma}$ consisting of eigenvectors of $|\mathfrak{L}^{g, g}_{\varsigma}|$ and $\{ u_n \}_n$ is an orthonormal set in $\ell^2(\Z^n)$. Then we have
\[ \sum_{n=1}^{\infty}  \langle \mathfrak{L}^{g, g}_{\varsigma} v_n, u_n \rangle_{\ell^2(\Z^n)}   = \sum_{n=1}^{\infty} s_n(\mathfrak{L}^{g, g}_{\varsigma})= \Vert \mathfrak{L}^{g, g}_{\varsigma} \Vert_{S_1}.  \]
Now, using Cauchy--Schwarz's inequality and Bessel's inequality, we get
\begin{eqnarray*}
\Vert \mathfrak{L}^{g, g}_{\varsigma} \Vert_{S_1} 
& = &  \sum_{n=1}^{\infty}  \langle \mathfrak{L}^{g, g}_{\varsigma} v_n, u_n \rangle_{\ell^2(\Z^n)} \\
& = & \sum_{n=1}^{\infty} \sum_{m \in \Z^n}  \int_{\T^n} \varsigma(m, w) \; V_{g} v_n (m,w) \; \overline{V_{g} u_n (m,w)} \; dw \\
& \leq & \sum_{m \in \Z^n}  \int_{\T^n} |\varsigma(m, w)| \left(\sum_{n=1}^{\infty} |V_{g} v_n (m,w)|^2 \right)^{1/2} \left(\sum_{n=1}^{\infty} |V_{g} u_n (m,w)|^2 \right)^{1/2} dw \\
& \leq & \Vert \varsigma \Vert_{L^1(\Z^n \times \T^n)} \;\Vert g \Vert^2_{\ell^2(\Z^n)}  \\
& \leq & \Vert \varsigma \Vert_{M^1(\Z^n \times \T^n)}\; \Vert g \Vert^2_{M^1(\Z^n)}.
\end{eqnarray*}
Next, we show that $\tilde{\varsigma} \in L^1(\Z^n \times \T^n)$.  Using formula (\ref{eq49}), we obtain 
\begin{eqnarray*}
|\tilde{\varsigma}(m, w)|
& = & \left|\left\langle \mathfrak{L}^{g, g}_{\varsigma}( M_w T_m g),  M_w T_m g \right\rangle_{\ell^2(\Z^n)}\right| \\
& = & \left| \sum_{n=1}^{\infty} s_n(\mathfrak{L}^{g, g}_{\varsigma}) \left\langle  M_w T_m g, v_n \right\rangle_{\ell^2(\Z^n)} \left\langle u_n,  M_w T_m g \right\rangle_{\ell^2(\Z^n)} \right| \\
& \leq & \frac{1}{2} \sum_{n=1}^{\infty} s_n(\mathfrak{L}^{g, g}_{\varsigma}) \left( \left| \left\langle  M_w T_m g, v_n \right\rangle_{\ell^2(\Z^n)}  \right|^2 + \left| \left\langle  M_w T_m g, u_n  \right\rangle_{\ell^2(\Z^n)} \right|^2 \right). 
\end{eqnarray*}  
Now, using Plancherel's formula (\ref{eq17}), we get 
\begin{eqnarray*}
&& \Vert \tilde{\varsigma} \Vert_{L^1(\Z^n \times \T^n)} 
= \sum_{m \in \Z^n}  \int_{\T^n} |\tilde{\varsigma}(m, w)| \; dw  \\
&& \leq \frac{1}{2} \sum_{n=1}^{\infty} s_n(\mathfrak{L}^{g, g}_{\varsigma}) \sum_{m \in \Z^n} \int_{\T^n} \left(  \left| \left\langle  M_w T_m g, v_n \right\rangle_{\ell^2(\Z^n)} \right|^2  + \left| \left\langle  M_w T_m g, u_n  \right\rangle_{\ell^2(\Z^n)} \right|^2  \right) dw  \\
&& \leq  \Vert g \Vert^2_{M^1(\Z^n)} \sum_{n=1}^{\infty} s_n(\mathfrak{L}^{g, g}_{\varsigma})  \\
&& = \Vert g \Vert^2_{M^1(\Z^n)} \; \Vert \mathfrak{L}^{g, g}_{\varsigma} \Vert_{S_1} .
\end{eqnarray*} 
This completes the proof of the theorem.
\end{proof}

\subsection{$M^p(\Z^n)$ Boundedness}
 
In this section, we prove that the localization operators $\mathfrak{L}^{g_1, g_2}_{\varsigma} : M^p(\Z^n) \to M^p(\Z^n)$ are bounded. We begin with the following propositions. 

\begin{proposition}\label{pro4} 
Let $\varsigma \in M^1(\Z^n \times \T^n)$, $g_1 \in  M^{p'}(\Z^n)$ and $g_2 \in  M^p(\Z^n)$, for $1 \leq p \leq \infty$. Then the localization operator $\mathfrak{L}^{g_1, g_2}_{\varsigma}: M^p(\Z^n) \to M^p(\Z^n)$ is a bounded linear operator, and we have
\[\Vert \mathfrak{L}^{g_1, g_2}_{\varsigma} \Vert_{\mathcal{B}(M^p(\Z^n))} \leq  \Vert \varsigma  \Vert_{M^1(\Z^n \times \T^n)} \; \Vert g_1 \Vert_{ M^{p'}(\Z^n)} \; \Vert g_2 \Vert_{M^p(\Z^n)}. \] 
\end{proposition}
\begin{proof}
Let $f \in M^p(\Z^n)$, $1 \leq p \leq \infty$ and $g \in M^{p'}(\Z^n)$. Then from H\"older's inequality, we have 
\begin{equation}\label{eq52}
\left| V_{g} f (m, w) \right| \leq \Vert f  \Vert_{M^p(\Z^n)} \; \Vert g  \Vert_{M^{p'}(\Z^n)}. 
\end{equation} 
For every $f \in M^p(\Z^n)$ and $h \in M^{p'}(\Z^n)$, using the relations (\ref{eq42}) and (\ref{eq52}), we obtain
\begin{eqnarray*}
\left| \left\langle \mathfrak{L}^{g_1, g_2}_{\varsigma}f, h \right\rangle \right| 
& \leq & \sum_{m \in \Z^n}  \int_{\T^n} |\varsigma(m, w)| \; \left|V_{g_1}f(m, w) \right| \left|   V_{g_2}h(m, w) \right| \; dw \\
&& \leq \Vert \varsigma \Vert_{M^1(\Z^n \times \T^n)} \; \Vert f \Vert_{M^p(\Z^n)} \;  \Vert g_1 \Vert_{M^{p'}(\Z^n)} \; \Vert h \Vert_{M^{p'}(\Z^n)} \; \Vert g_2 \Vert_{M^p(\Z^n)}. 
\end{eqnarray*}
Hence,
\[\Vert \mathfrak{L}^{g_1, g_2}_{\varsigma} \Vert_{\mathcal{B}(M^p(\Z^n))} \leq  \Vert \varsigma  \Vert_{M^1(\Z^n \times \T^n)} \; \Vert g_1 \Vert_{ M^{p'}(\Z^n)} \; \Vert g_2 \Vert_{M^p(\Z^n)}. \] 
\end{proof}

Next, we obtain an $M^p(\Z^n)$-boundedness result using the Schur technique. The estimate obtained for the norm $\Vert \mathfrak{L}^{g_1, g_2}_{\varsigma} \Vert_{\mathcal{B}(M^p(\Z^n))}$ is different from the previous Proposition. 

\begin{proposition}\label{pro5}
Let $\varsigma \in M^1(\Z^n \times \T^n)$ and $g_1, g_2 \in  M^1(\Z^n) \cap \ell^\infty(\Z^n)$. Then there exists a bounded linear operator $\mathfrak{L}^{g_1, g_2}_{\varsigma}: M^p(\Z^n) \to  M^p(\Z^n) $, $1 \leq p \leq \infty$ such that
\begin{eqnarray*}
\Vert \mathfrak{L}^{g_1, g_2}_{\varsigma} \Vert_{\mathcal{B}(M^p(\Z^n))} 
\leq \max(\Vert g_1 \Vert_{M^1(\Z^n)} \Vert g_2 \Vert_{\ell^\infty(\Z^n)}, \Vert g_1 \Vert_{\ell^\infty(\Z^n)} \Vert g_2 \Vert_{M^1(\Z^n)}) \; \Vert \varsigma \Vert_{M^1(\Z^n \times \T^n)}.
\end{eqnarray*}
\end{proposition}
\begin{proof}
Let $\mathcal{K}$ be the function defined on $\Z^n \times \Z^n$ by 
\begin{equation}\label{eq53}
\mathcal{K}(k, l) = \sum_{m \in \Z^n}  \int_{\T^n} \varsigma(m, w)\; \overline{M_w T_m g_1(l)} \; M_w T_m g_2(k) \; dw.
\end{equation}
Then we define 
\[ \mathfrak{L}^{g_1, g_2}_{\varsigma}f(k)= \sum_{l \in \Z^n} \mathcal{K}(k,l) \; f(l). \]
Now, for any $l \in \Z^n$, we obtain
\begin{eqnarray*}
\sum_{k \in \Z^n} | \mathcal{K}(k, l) | 
&\leq & \sum_{k \in \Z^n} \sum_{m \in \Z^n}  \int_{\T^n} |\varsigma(m, w)| \left|\overline{ M_w T_m g_1(l)} \right| \; \left|M_w T_m g_2(k) \right| \; dw  \\
&& \leq \Vert g_1 \Vert_{\ell^\infty(\Z^n)} \; \Vert g_2 \Vert_{M^1(\Z^n)} \; \Vert \varsigma \Vert_{M^1(\Z^n \times \T^n)}, 
\end{eqnarray*}
and for any $k \in \Z^n$, we obtain
\[ \sum_{l \in \Z^n} | \mathcal{K}(k, l) | \leq \Vert g_1 \Vert_{M^1(\Z^n)} \; \Vert g_2 \Vert_{\ell^\infty(\Z^n)} \; \Vert \varsigma \Vert_{M^1(\Z^n \times \T^n)}.   \]
Thus using Schur's lemma (see \cite{fol95}), we conclude that, for $1 \leq p \leq \infty$, $\mathfrak{L}^{g_1, g_2}_{\varsigma}: M^p(\Z^n) \to  M^p(\Z^n) $ is a bounded linear operator, and we have 
\begin{eqnarray*}
\Vert \mathfrak{L}^{g_1, g_2}_{\varsigma} \Vert_{\mathcal{B}(M^p(\Z^n))} 
\leq \max(\Vert g_1 \Vert_{M^1(\Z^n)} \Vert g_2 \Vert_{\ell^\infty(\Z^n)}, \Vert g_1 \Vert_{\ell^\infty(\Z^n)} \Vert g_2 \Vert_{M^1(\Z^n)}) \; \Vert \varsigma \Vert_{M^1(\Z^n \times \T^n)}.
\end{eqnarray*}
\end{proof} 
 
\begin{remark}
From the Proposition \ref{pro5}, we conclude that the bounded linear operator on $M^p(\Z^n)$, $1 \leq p \leq \infty$, obtained in Proposition \ref{pro4} is actually the discrete integral operator on $M^p(\Z^n)$ with the kernel $\mathcal{K}$ given by (\ref{eq53}).
\end{remark}

\begin{theorem}\label{th3}
Let $\varsigma \in L^1(\Z^n \times \T^n)$, $g_1 \in \ell^2(\Z^n)$ and $g_2 \in  M^p(\Z^n)$, for $1 \leq p \leq 2$. Then the localization operator $ \mathfrak{L}^{g_1, g_2}_{\varsigma}$ is in $\mathcal{B}(M^p(\Z^n ))$, and we have
\[ \Vert \mathfrak{L}^{g_1, g_2}_{\varsigma} \Vert_{\mathcal{B}(M^p(\Z^n ))} \leq \Vert \varsigma \Vert_{L^1(\Z^n \times \T^n)} \; \Vert g_1 \Vert_{\ell^2(\Z^n)} \; \Vert g_2 \Vert_{M^p(\Z^n)}. \]
\end{theorem}
\begin{proof}
Let $f \in M^p(\Z^n)$ and $h \in M^{p'}(\Z^n)$, where $p'$ is the conjugate exponent of $p$. Using the duality between the modulation spaces $ M^p(\Z^n)$ and $M^{p'}(\Z^n)$, we obtain 
\begin{eqnarray*}
\left| \left\langle \mathfrak{L}^{g_1, g_2}_{\varsigma}f, h \right\rangle \right| 
& \leq & \sum_{m \in \Z^n}  \int_{\T^n} |\varsigma(m, w)| \; \left|V_{g_1}f(m, w) \right| \; \left| V_{g_2}h(m, w) \right| \; dw  \\ 
& \leq & \left\Vert \varsigma \right\Vert_{L^1(\Z^n \times \T^n)} \Vert f \Vert_{M^p(\Z^n)} \Vert g_1 \Vert_{M^{p'}(\Z^n)} \Vert h \Vert_{M^{p'}(\Z^n)} \Vert g_2 \Vert_{M^p(\Z^n)} .
\end{eqnarray*}
Since $1 \leq p \leq 2$, we have $p' \geq 2 $ and $\ell^2(\Z^n) \subset M^{p'}(\Z^n)$. Using $\Vert g_1 \Vert_{M^{p'}(\Z^n)} \leq \Vert g_1 \Vert_{\ell^2(\Z^n)}$, we obtain
\[ \Vert \mathfrak{L}^{g_1, g_2}_{\varsigma} \Vert_{\mathcal{B}(M^p(\Z^n ))} \leq \Vert \varsigma \Vert_{L^1(\Z^n \times \T^n)} \; \Vert g_1 \Vert_{\ell^2(\Z^n)} \; \Vert g_2 \Vert_{M^p(\Z^n)}. \]
This completes the proof of the theorem.
\end{proof}

\subsection{Compactness of $\mathfrak{L}^{g_1, g_2}_{\varsigma}$ for symbols in $M^1(\Z^n \times \T^n)$}

In this section, we prove that the localization operators $\mathfrak{L}^{g_1, g_2}_{\varsigma}: M^p(\Z^n) \to M^p(\Z^n), \; 1 < p< \infty$ are compact for symbols $\varsigma$ in $M^1(\Z^n \times \T^n)$. Let us start with the following proposition.

\begin{proposition}
Let $\varsigma \in M^1(\Z^n \times \T^n)$ and $g_1, g_2 \in  M^1(\Z^n)$. Then, for fixed $2 \leq p< \infty$, the localization operator $\mathfrak{L}^{g_1, g_2}_{\varsigma}: M^p(\Z^n) \to M^p(\Z^n)$ is compact.
\end{proposition}
\begin{proof}
Let $\{ f_n \}_{n \in \mathbb{N}}$ be a sequence of functions in $M^p(\Z^n)$ such that $f_n \rightharpoonup 0$ weakly in $M^p(\Z^n)$ as $n \to \infty$. It is sufficient to prove that $\lim\limits_{n \to \infty} \Vert \mathfrak{L}^{g_1, g_2}_{\varsigma}f_n \Vert_{M^p(\Z^n)}=0$. From the relation (\ref{eq41}), we obtain 
\begin{eqnarray}\label{eq54}
\left| \mathfrak{L}^{g_1, g_2}_{\varsigma} f_n(k) \right|
\leq \sum_{m \in \Z^n}  \int_{\T^n} |\varsigma(m, w)| \left| \left\langle f_n,   M_w T_m g_1 \right\rangle_{\ell^2(\Z^n)} \right| \; \left| M_w T_m g_2(k) \right| \; dw.
\end{eqnarray}
Since $f_n \rightharpoonup 0$ weakly in $M^p(\Z^n)$, we deduce that
\begin{equation}\label{eq55}
\lim_{n \to \infty} |\varsigma(m, w)| \left| \left\langle f_n, M_w T_m g_1 \right\rangle_{\ell^2(\Z^n)} \right| \; \left| M_w T_m g_2(k) \right|=0, \quad  \text{for all} \;\; m, k \in \Z^n \text{\;and\;} w \in \T^n.
\end{equation}
Moreover, as $f_n \rightharpoonup 0$ weakly in $M^p(\Z^n)$ as $n \to \infty$, then there exists a constant $C>0$ such that $\Vert f_n \Vert_{M^p(\Z^n)} \leq C$. Hence, for all  $m, k \in \Z^n \text{\;and\;} w \in \T^n$, we obtain 
\begin{equation}\label{eq56}
|\varsigma(m, w)| \left| \left\langle f_n, M_w T_m g_1 \right\rangle_{\ell^2(\Z^n)} \right|  \left| M_w T_m g_2(k) \right|  \leq C |\varsigma(m, w)| \; \Vert g_1 \Vert_{M^1(\Z^n)} \left| M_w T_m g_2(k) \right|.
\end{equation}
Since $p \geq 2$, using $\ell^1(\Z^n) \subset \ell^2(\Z^n) \subset M^p(\Z^n)$, we get
\begin{eqnarray}\label{eq57}
\Vert \mathfrak{L}^{g_1, g_2}_{\varsigma} f_n \Vert_{M^p(\Z^n)}  
& \leq & \left\|  \mathfrak{L}^{g_{1}, g_{2}}_{\varsigma} f_{n}  \right\|_{\ell^{1}\left( \Z^n \right)}
\nonumber \\
& \leq &  \sum_{k \in \Z^n}  \sum_{m \in \Z^n}  \int_{\T^n} |\varsigma(m, w)| \left| \left\langle f_n, M_w T_m g_1  \right\rangle_{\ell^2(\Z^n)} \right|  \left| M_w T_m g_2(k) \right| \; dw  \nonumber \\
& \leq & C \; \Vert g_1 \Vert_{M^1(\Z^n)} \sum_{m \in \Z^n}  \int_{\T^n} |\varsigma(m, w)|  \sum_{k \in \Z^n} \left|  M_w T_m g_2(k) \right| \; dw \nonumber \\
& \leq & C \; \Vert g_1 \Vert_{M^1(\Z^n)} \; \Vert g_2 \Vert_{M^1(\Z^n)} \; \Vert \varsigma \Vert_{M^1(\Z^n \times \T^n)} < \infty. 
\end{eqnarray}
Thus, using the Lebesgue dominated convergence theorem and the relations (\ref{eq54}) -- (\ref{eq57}), we obtain that 
\[\lim\limits_{n \to \infty} \Vert \mathfrak{L}^{g_1, g_2}_{\varsigma} f_n \Vert_{M^p(\Z^n)}=0.  \]
This completes the proof. 
\end{proof}

\begin{theorem}
Let $\varsigma \in M^1(\Z^n \times \T^n)$,  $g_1 \in  M^{p'}(\Z^n)$ and $g_2 \in  M^p(\Z^n)$, for $1 < p < \infty$. Then the localization operator $\mathfrak{L}^{g_1, g_2}_{\varsigma}: M^p(\Z^n) \to M^p(\Z^n)$ is compact.
\end{theorem}
\begin{proof}
Let $p'$ be the conjugate exponent of $p$. We first show that the conclusion of the previous proposition holds for $p'$. The operator $\mathfrak{L}^{g_1, g_2}_{\varsigma}: M^{p'}(\Z^n) \to M^{p'}(\Z^n)$ is the adjoint of the operator $\mathfrak{L}^{g_2, g_1}_{\overline{\varsigma}}: M^p(\Z^n) \to M^p(\Z^n)$, which is compact by the previous proposition for fixed $2 \leq p < \infty$. Hence, by the duality properties of modulation spaces, $\mathfrak{L}^{g_1, g_2}_{\varsigma}: M^{p'}(\Z^n) \to M^{p'}(\Z^n)$ is compact for fixed $1 < p' \leq 2$. Finally, using the interpolation of the compactness on $M^p(\Z^n)$ and on $M^{p'}(\Z^n)$, the proof is complete. 
\end{proof}

\section{The Landau--Pollak--Slepian Operator}\label{sec5}

In this section, we show that the Landau--Pollak--Slepian type operator is actually a localization operator. We first need to introduce the following operators to define the Landau--Pollak--Slepian operator on $\Z^n$. 

Let $\Omega$ and $T$ be two positive numbers. Then we define the operators $P_{\Omega} : \ell^2(\Z^n ) \to \ell^2(\Z^n )$ and $Q_T : \ell^2(\Z^n) \to \ell^2(\Z^n)$ by 
\begin{equation}\label{eq19}
(P_{\Omega} f)^{\wedge}(w)= \left\{\begin{array}{cc}
\hat{f}(w), & |w| \leq \Omega, \\
0, & |w|>\Omega,
\end{array}\right.
\end{equation}
where $w \in \T^n$, $|w|=|w_1|+|w_2|+ \cdots + |w_n|$ and 
\begin{equation}\label{eq20}
\left(Q_{T} f\right)(k)= \left\{\begin{array}{cc}
f(k), & |k| \leq T, \\
0, & |k|>T,
\end{array}\right.
\end{equation}
where $k \in \Z^n$, $|k|=|k_1|+|k_2|+ \cdots + |k_n|$ and $f \in \ell^{2}(\Z^{n})$. 
\begin{proposition}
The operators $P_{\Omega}$ and $Q_T$ are self-adjoint projections.
\end{proposition}
\begin{proof}
For $f, h \in \ell^2(\Z^n)$, using Plancherel's formula, we get
\begin{eqnarray*}
\left\langle P_{\Omega} f, h \right\rangle_{\ell^2(\Z^n)}
& = & \left\langle  (P_{\Omega} f )^{\wedge}, \hat{h} \right\rangle_{L^2(\T^n)}
= \int_{\T^n} (P_{\Omega} f)^{\wedge}(w) \; \overline{\hat{h}(w)} \; dw
= \int_{B_\Omega} \hat{f}(w) \; \overline{\hat{h}(w)} \; dw\\
& = & \int_{B_\Omega} \hat{f}(w) \; \overline{(P_{\Omega} h)^{\wedge}(w)} \; dw 
= \int_{\T^n} \hat{f}(w) \; \overline{(P_{\Omega} h)^{\wedge}(w)} \; dw \\
& = & \left\langle \hat{f}, (P_{\Omega} h )^{\wedge} \right\rangle_{L^2(\T^n)}
= \left\langle f, P_{\Omega} h \right\rangle_{\ell^2(\Z^n)},
\end{eqnarray*}
where $B_{\Omega} := \left\{ w \in \T^n: \;|w| \leq \Omega \right\}$. Hence the operator $P_{\Omega}$ is self-adjoint. Also, for $f, h \in \ell^2(\Z^n)$, we have 
\begin{eqnarray*}
\left\langle Q_{T} f, h \right\rangle_{\ell^2(\Z^n)}
= \sum_{k \in \Z^n} (Q_{T} f)(k) \; \overline{h(k)} 
= \sum_{k \in B_T}  f(k) \; \overline{h(k)}
= \sum_{k \in \Z^n} f(k) \; \overline{(Q_{T} h)(k)}
= \left\langle f, Q_{T} h \right\rangle_{\ell^2(\Z^n)}, 
\end{eqnarray*}
where $B_{T}:= \left\{ k \in \Z^n: \;|k| \leq T \right\}$.  Hence the operator $Q_{T}$ is self-adjoint. Since the operator $P_{\Omega}$ is self-adjoint, using Plancherel's formula, we obtain
\begin{eqnarray*}
\left\langle P^2_{\Omega} f, h \right\rangle_{\ell^2(\Z^n)} 
& = & \left\langle P_{\Omega} f, P_{\Omega} h \right\rangle_{\ell^2(\Z^n)}
= \left\langle (P_{\Omega} f )^{\wedge}, (P_{\Omega} h )^{\wedge} \right\rangle_{L^2(\T^n)}
=\int_{\T^n} (P_{\Omega} f )^{\wedge}(w) \; \overline{(P_{\Omega} h )^{\wedge}(w)} \; dw \\
&=& \int_{B_{\Omega}} \hat{f}(w) \; \overline{\hat{h}(w)} \; dw 
= \int_{\T^n} (P_{\Omega} f )^{\wedge}(w) \; \overline{\hat{h}(w)} \; dw \\
& = & \left\langle (P_{\Omega} f )^{\wedge}, \hat{h} \right\rangle_{L^2(\T^n)}
= \left\langle P_{\Omega}f, h \right\rangle_{\ell^2(\Z^n )},
\end{eqnarray*}
for all $f, h \in \ell^2(\Z^n)$. Thus, $P^2_{\Omega} = P_{\Omega}$ and therefore $P_{\Omega}$ is a projection. Finally, using the fact that the operator $Q_T$ is self-adjoint, we get
\begin{eqnarray*}
\left\langle Q^2_T f, h \right\rangle_{\ell^2(\Z^n)} 
& = & \left\langle Q_T f, Q_T h \right\rangle_{\ell^2(\Z^n)}
=\sum_{k \in \Z^n} (Q_T f)(k) \; \overline{(Q_T h)(k)}
=\sum_{k \in B_T} f(k) \; \overline{h(k)} \\
& = & \sum_{k \in \Z^n} (Q_T f)(k) \; \overline{h(k)}
= \left\langle Q_T f, h \right\rangle_{\ell^2(\Z^n)} ,
\end{eqnarray*}
for all $f, h \in \ell^2(\Z^n)$. Thus, $Q^2_T = Q_T$ and therefore $Q_T$ is a projection. 
\end{proof} 

The linear operator $P_{\Omega} Q_T  P_{\Omega} : \ell^2(\Z^n ) \to \ell^2(\Z^n)$ is called the Landau--Pollak--Slepian operator on $\Z^n$. Next, we show that the operator $Q_{2T} P_{\Omega} Q_{2T}$ is actually a localization operator on $\Z^n$. 

\begin{theorem}
Let $g_1$ and $g_2$ be two functions on $\Z^n$ defined by 
\[g_1(k)=g_2(k) =\frac{1}{\sqrt{\mathrm{card}(B_T)}} \chi_{B_T}(k), \quad  k \in \Z^n,\]
where $\mathrm{card}(B_T)$ is the cardinality of $B_T$ and
$\chi_{B_T}$ is the characteristic function on $B_T$ defined by 
\[ \chi_{B_T}(k)= 
\begin{cases}
1, & |k| \leq T, \\ 
0, & |k|>T.
\end{cases}
\]
Let $\varsigma$ be the characteristic function on $B_T \times B_{\Omega}$, given by  
\[ \varsigma(m, w)= 
\begin{cases}
1, &  m \in B_T \; \mathrm{and} \;  w \in B_{\Omega}, \\ 
0, & \mathrm{otherwise}. 
\end{cases}
\]
Then the operator $Q_{2T} P_{\Omega} Q_{2T}: \ell^{2}(\Z^n) \to \ell^{2}(\Z^n)$ is equal to the localization operator $ \mathfrak{L}^{g_1, g_2}_{\varsigma} : \ell^2(\Z^n) \to \ell^2(\Z^n)$.  In fact, $ Q_{2T} P_{\Omega} Q_{2T} =\mathfrak{L}^{g_1, g_2}_{\varsigma}. $


\end{theorem} 

\begin{proof}
We have 
\[\Vert g_1 \Vert^2_{\ell^2(\Z^n)}= \Vert g_2 \Vert^2_{\ell^2(\Z^n)}= \frac{1}{\mathrm{card}(B_T)} \sum_{k \in \Z^n}  \left|\chi_{B_T}(k) \right|^2
= \frac{1}{\mathrm{card}(B_T)} \sum_{k \in B_T} 1 
=1. \]
Also, for all $f, h \in \ell^2(\Z^n)$, we have 
\begin{eqnarray*}
\left\langle \mathfrak{L}^{g_1, g_2}_{\varsigma}f, h \right\rangle_{\ell^2(\Z^n)}
& = & \sum_{m \in \Z^n}  \int_{\T^n} \varsigma(m, w) \; V_{g_1}f(m,w) \; \overline{V_{g_2}h(m,w)} \; dw \\
& = & \sum_{m \in B_T}  \int_{B_{\Omega}} V_{g_1}f(m,w) \; \overline{V_{g_2}h(m,w)} \; dw. 
\end{eqnarray*}
Now, we compute
\begin{eqnarray*}
V_{g_1}f(m,w) 
& = & \langle f, M_w T_m g_1 \rangle_{\ell^2(\Z^n)}
= \langle \hat{f}, \widehat{M_w T_m g_1} \rangle_{L^2(\T^n)} 
= \langle \hat{f}, T_w M_{-m} \hat{g_1} \rangle_{L^2(\T^n)} \\
& = & \int_{\T^n} \hat{f}(\xi) \; \overline{T_w M_{-m} \hat{g_1}(\xi)} \; d\xi 
= \int_{\T^n} \hat{f}(\xi) \; \overline{M_{-m} \hat{g_1}(\xi-w)} \; d\xi \\
& = & \int_{\T^n} \hat{f}(\xi) \;e^{2 \pi i m \cdot (\xi-w)} \overline{\hat{g_1}(\xi-w)} \; d\xi \\
& = & \int_{\T^n} \hat{f}(\xi) \;e^{- 2 \pi i m \cdot w} \;e^{2 \pi i m \cdot \xi}\; \frac{1}{\sqrt{\mathrm{card}(B_T)}} \sum_{k \in B_T} e^{2 \pi i k \cdot (\xi-w)} \; d\xi \\
& = & \frac{1}{\sqrt{\mathrm{card}(B_T)}} \;e^{- 2 \pi i m \cdot w} \sum_{k \in B_T} e^{- 2 \pi i k \cdot w} \int_{\T^n} \hat{f}(\xi) \;e^{2 \pi i (m+k) \cdot \xi}  \; d\xi \\
& = & \frac{1}{\sqrt{\mathrm{card}(B_T)}} \;e^{- 2 \pi i m \cdot w} \sum_{k \in B_T} e^{- 2 \pi i k \cdot w} \; f(m+k)  \\
& = & \frac{1}{\sqrt{\mathrm{card}(B_T)}} \;e^{- 2 \pi i m \cdot w} \sum_{k \in \Z^n} e^{- 2 \pi i k \cdot w} \; (T_{-m}Q_{2T} f)(k)  \\
& = & \frac{1}{\sqrt{\mathrm{card}(B_T)}} \;e^{- 2 \pi i m \cdot w} \;(T_{-m}Q_{2T} f)^{\wedge}(w) \\
& = & \frac{1}{\sqrt{\mathrm{card}(B_T)}} \;e^{- 2 \pi i m \cdot w} \; M_{m}(Q_{2T} f)^{\wedge}(w) \\
& = & \frac{1}{\sqrt{\mathrm{card}(B_T)}} \; (Q_{2T} f )^{\wedge}(w).
\end{eqnarray*}
Similarly, we get 
\[V_{g_2}h(m,w)=\frac{1}{\sqrt{\mathrm{card}(B_T)}} \; (Q_{2T} h )^{\wedge}(w).  \] 
Hence, using Plancherel's formula and the fact that the operators $P_{\Omega}$ and $Q_{2T}$ are self-adjoint projections, we obtain 
\begin{eqnarray*}
\left\langle \mathfrak{L}^{g_1, g_2}_{\varsigma}f, h \right\rangle_{\ell^2(\Z^n)} 
& = & \frac{1}{\mathrm{card}(B_T)} \sum_{m \in B_T}  \int_{B_{\Omega}}  (Q_{2T} f )^{\wedge}(w) \; \overline{(Q_{2T} h )^{\wedge}(w)} \; dw \\
& = & \int_{B_{\Omega}}  (Q_{2T} f )^{\wedge}(w) \; \overline{(Q_{2T} h )^{\wedge}(w)} \; dw \\
& = & \int_{\T^n} (P_{\Omega} Q_{2T} f )^{\wedge}(w) \; \overline{(P_{\Omega} Q_{2T} h )^{\wedge}(w)} \; dw \\
& = & \left\langle (P_{\Omega} Q_{2T} f )^{\wedge}, (P_{\Omega} Q_{2T} h )^{\wedge} \right\rangle_{L^2(\T^n)}
= \left\langle P_{\Omega} Q_{2T} f, P_{\Omega} Q_{2T} h \right\rangle_{\ell^2(\Z^n)} \\
& = & \left\langle P^2_{\Omega} Q_{2T} f, Q_{2T} h \right\rangle_{\ell^2(\Z^n)} 
= \left\langle Q_{2T} P_{\Omega} Q_{2T} f, h \right\rangle_{\ell^2(\Z^n)} ,
\end{eqnarray*}
for all $f, h \in \ell^2(\Z^n)$, and therefore we have the desired result.
\end{proof}

\section{Localization operators for different types of symbols}\label{sec678}

\subsection{Paracommutators}\label{sec6}
  
Let $\varsigma$ be a function on $\Z^n \times \T^n$ given by
\[ \varsigma(m, w)=\alpha(m) \beta(w),  \quad (m, w) \in \Z^n \times \T^n, \]
where $\alpha$ and $\beta$ are suitable functions on $\Z^n$ and $\T^n$ respectively. Then for all $f, h \in \ell^2(\Z^n)$, using Plancherel's formula, we obtain
\begin{eqnarray}
&& \left\langle \mathfrak{L}^{g_1, g_2}_{\varsigma}f, h \right\rangle_{\ell^2(\Z^n)} 
= \sum_{m \in \Z^n}  \int_{\T^n} \varsigma(m, w) \; V_{g_1}f(m,w) \;  \overline{V_{g_2}h(m,w)} \; dw  \nonumber \\
&& = \sum_{m \in \Z^n}  \int_{\T^n} \alpha(m) \beta(w)\; \langle f, M_w T_m g_1 \rangle_{\ell^2(\Z^n)} \;
\overline{\langle h, M_w T_m g_2 \rangle}_{\ell^2(\Z^n)} \; dw \nonumber \\
&& = \sum_{m \in \Z^n} \alpha(m) \int_{\T^n} \beta(w)\; \langle \hat{f}, \widehat{M_w T_m g_1} \rangle_{L^2(\T^n)} \; \overline{\langle \hat{h}, \widehat{M_w T_m g_2} \rangle}_{L^2(\T^n)} \; dw \nonumber \\
&& = \sum_{m \in \Z^n} \alpha(m) \int_{\T^n} \beta(w)\;\langle \hat{f}, T_w M_{-m} \hat{g_1} \rangle_{L^2(\T^n)} \; \overline{\langle \hat{h}, T_w M_{-m}\hat{g_2} \rangle}_{L^2(\T^n)} \; dw \nonumber \\
&& = \sum_{m \in \Z^n} \alpha(m) \int_{\T^n} \beta(w)  \left( \int_{\T^n} \hat{f}(\xi) \; \overline{T_w M_{-m} \hat{g_1}(\xi)} \; d\xi \right) \left(\int_{\T^n} \overline{\hat{h}(\eta)} \; T_w M_{-m}\hat{g_2}(\eta) \; d\eta \right) dw \nonumber \\
&& = \sum_{m \in \Z^n} \alpha(m) \int_{\T^n} \beta(w) \left( \int_{\T^n} \hat{f}(\xi) \; \overline{M_{-m} \hat{g_1}(\xi-w)} \; d\xi \right) \left(\int_{\T^n} \overline{\hat{h}(\eta)} \;  M_{-m}\hat{g_2}(\eta-w) \; d\eta \right) dw \nonumber \\
&& = \sum_{m \in \Z^n} \alpha(m) \int_{\T^n} \beta(w) \left( \int_{\T^n} \hat{f}(\xi) \;e^{2 \pi i m \cdot (\xi-w)} \overline{\hat{g_1}(\xi-w)} \; d\xi \right) \left(\int_{\T^n} \overline{\hat{h}(\eta)} \; e^{-2 \pi i m \cdot (\eta-w)} \hat{g_2}(\eta-w) \; d\eta \right) dw \nonumber \\
&& = \sum_{m \in \Z^n} \alpha(m) \int_{\T^n} \beta(w) \left( \int_{\T^n} \hat{f}(\xi) \;e^{2 \pi i m \cdot \xi} \overline{\hat{g_1}(\xi-w)} \; d\xi \right) \left(\int_{\T^n} \overline{\hat{h}(\eta)} \; e^{-2 \pi i m \cdot \eta} \hat{g_2}(\eta-w) \; d\eta \right) dw \nonumber \\
&& = \int_{\T^n} \int_{\T^n} \int_{\T^n} \left(\sum_{m \in \Z^n} \alpha(m) \; e^{-2 \pi i m \cdot (\eta - \xi)} \right) \beta(w) \hat{f}(\xi) \; \overline{\hat{h}(\eta)} \; \overline{\hat{g_1}(\xi-w)} \; \hat{g_2}(\eta-w) \; d\xi \; d\eta \; dw \nonumber
\end{eqnarray}
\begin{eqnarray}\label{eq21}
&& = \int_{\T^n} \int_{\T^n} \int_{\T^n} \hat{\alpha}(\eta - \xi) \; \beta(w) \hat{f}(\xi) \; \overline{\hat{h}(\eta)} \; \overline{\hat{g_1}(\xi-w)} \; \hat{g_2}(\eta-w) \; d\xi \; d\eta \; dw \nonumber \\
&& = \int_{\T^n} \int_{\T^n} A(\xi, \eta) \; \hat{\alpha}(\eta - \xi) \;  \hat{f}(\xi) \; \overline{\hat{h}(\eta)} \;  d\xi \; d\eta, 
\end{eqnarray}
where
\[A(\xi, \eta)=\int_{\T^n} \beta(w) \; \overline{\hat{g_1}(\xi-w)} \; \hat{g_2}(\eta-w) \; dw,\]
for all $\xi, \eta \in \T^n$. Thus, the localization operator $\mathfrak{L}^{g_1, g_2}_{\varsigma}$ is a paracommutator with Fourier kernel $A$ and symbol $\alpha$. 

\subsection{A Paraproduct Connection}\label{sec7}

Here, we consider the case when the symbol $\varsigma$ is independent of the second variable, i.e.,   
\[ \varsigma(m, w)=\alpha(m),  \quad (m, w) \in \Z^n \times \T^n, \]
where $\alpha$ is a function on $\Z^n$. Let $\tilde{g_1}$ be the function on $\Z^n$ defined by $\tilde{g_1}(k)=\overline{g_1(-k)}$, $k \in \Z^n$. Now, for all $f, h \in \ell^2(\Z^n)$, using Plancherel's formula as in the preceding subsection, we get  
\begin{eqnarray*}
&& \left\langle \mathfrak{L}^{g_1, g_2}_{\varsigma}f, h \right\rangle_{\ell^2(\Z^n)} 
= \int_{\T^n} \int_{\T^n} \int_{\T^n} \hat{\alpha}(\eta - \xi) \; \hat{f}(\xi) \; \overline{\hat{h}(\eta)} \; \overline{\hat{g_1}(\xi-w)} \; \hat{g_2}(\eta-w) \; d\xi \; d\eta \; dw \\
&&= \int_{\T^n} \left(\int_{\T^n} \int_{\T^n} \hat{\alpha}(\eta - \xi) \; \hat{f}(\xi) \; T_w \hat{\tilde{g_1}}(\xi) \; T_w\hat{g_2}(\eta) \; d\xi  \; dw \right) \overline{\hat{h}(\eta)} \; d\eta \\ 
&& = \int_{\T^n} \left(\int_{\T^n} \int_{\T^n} \hat{\alpha}(\eta - \xi) \; \hat{f}(\xi) \; (M_w \tilde{g_1})^{\wedge}(\xi) \; (M_w g_2)^{\wedge}(\eta) \; d\xi  \; dw \right) \overline{\hat{h}(\eta)} \; d\eta \\ 
&& = \int_{\T^n} \int_{\T^n} \left( \int_{\T^n} \hat{\alpha}(\eta - \xi) \; \hat{f}(\xi) \; (M_w \tilde{g_1})^{\wedge}(\xi) \; (M_w g_2)^{\wedge}(\eta) \; d\xi  \right) dw \; \overline{\hat{h}(\eta)} \; d\eta \\
&& = \int_{\T^n} \int_{\T^n} \left( \int_{\T^n} \hat{\alpha}(\eta - \xi) \; (M_w \tilde{g_1} * f)^{\wedge}(\xi) \; (M_w g_2)^{\wedge}(\eta) \; d\xi  \right) dw \; \overline{\hat{h}(\eta)} \; d\eta \\
&& = \int_{\T^n} \int_{\T^n}  (\hat{\alpha} * (M_w \tilde{g_1} * f)^{\wedge}) (\eta) \;(M_w g_2)^{\wedge}(\eta) \; \overline{\hat{h}(\eta)} \; d\eta \; dw \\
&& = \int_{\T^n} \int_{\T^n} (\alpha (M_w \tilde{g_1} * f))^{\wedge} (\eta) \;(M_w g_2)^{\wedge}(\eta) \; \overline{\hat{h}(\eta)} \; d\eta \; dw \\
&& = \int_{\T^n} \int_{\T^n} (\alpha (M_w \tilde{g_1} * f)*M_w g_2)^{\wedge} (\eta) \; \overline{\hat{h}(\eta)} \; d\eta \; dw \\
&& = \int_{\T^n} \left\langle (\alpha (M_w \tilde{g_1} * f)*M_w g_2)^{\wedge}, \hat{h} \right\rangle_{L^2(\T^n)}  \; dw \\
&& = \int_{\T^n} \left\langle (\alpha (M_w \tilde{g_1} * f)*M_w g_2),  h \right\rangle_{\ell^2(\Z^n)}  \; dw \\
&& = \left\langle \int_{\T^n} (\alpha (M_w \tilde{g_1} * f)*M_w g_2) \; dw ,  h \right\rangle_{\ell^2(\Z^n)}.
\end{eqnarray*}
Hence
\[ \mathfrak{L}^{g_1, g_2}_{\varsigma}f(k)= \int_{\T^n} (\alpha (M_w \tilde{g_1} * f)*M_w g_2)(k) \; dw, \quad k \in \Z^n.  \]  
Further,          
\begin{eqnarray*}
&& \left\langle \mathfrak{L}^{g_1, g_2}_{\varsigma}f, h \right\rangle_{\ell^2(\Z^n)} 
= \left\langle \int_{\T^n} (\alpha (M_w \tilde{g_1} * f)*M_w g_2) \; dw ,  h \right\rangle_{\ell^2(\Z^n)} \\
&& = \sum_{k \in \Z^n} \int_{\T^n} (\alpha (M_w \tilde{g_1} * f)*M_w g_2)(k) \; dw \; \overline{h(k)} 
= \sum_{k \in \Z^n} \alpha(k)\; p_{g_1, g_2}(f,h)(k),
\end{eqnarray*}
for all $f, h \in \ell^2(\Z^n)$, where $p_{g_1, g_2}(f,h)$ is the paraproduct of $f$ and $h$ with respect
to the window functions $g_1$ and $g_2$ given by
\[ p_{g_1, g_2}(f,h)(k)= \int_{\T^n} (M_w \tilde{g_1} * f)(k) \; \overline{(\overline{M_w g_2(- \; \cdot)} *h)(k)}  \; dw, \quad k \in \Z^n. \]
Next, we give an $\ell^1$-estimate on the paraproduct $p_{g_1, g_2}(f,h)$, where $f, h \in \ell^2(\Z^n)$, and the window functions $g_1$ and $g_2$ are also in $\ell^2(\Z^n)$. We have
\begin{eqnarray*}
&& \sum_{k \in \Z^n} p_{g_1, g_2}(f,h)(k) 
= \sum_{k \in \Z^n} \int_{\T^n} (M_w \tilde{g_1} * f)(k) \; \overline{(\overline{M_w g_2(- \; \cdot)} *h)(k)}  \; dw \\
&& = \int_{\T^n} \left\langle (M_w \tilde{g_1} * f),  (\overline{M_w g_2(- \; \cdot)} *h) \right\rangle_{\ell^2(\Z^n)} dw \\
&&= \int_{\T^n} \left\langle (M_w \tilde{g_1} * f)^{\wedge},  (\overline{M_w g_2(- \; \cdot)} *h)^{\wedge} \right\rangle_{L^2(\T^n)} dw \\
&&= \int_{\T^n} \int_{\T^n} (M_w \tilde{g_1} * f)^{\wedge}(\xi) \; \overline{(\overline{M_w g_2(- \; \cdot)} *h)^{\wedge}(\xi)}  \;  d\xi \; dw \\
&&= \int_{\T^n} \int_{\T^n} (M_w \tilde{g_1})^{\wedge}(\xi) \; \hat{f}(\xi) \; \overline{(\overline{M_w g_2(- \; \cdot)})^{\wedge}(\xi)} \; \overline{\hat{h}(\xi)} \;  d\xi \; dw \\
&&= \int_{\T^n} \int_{\T^n} (M_w \tilde{g_1})^{\wedge}(\xi) \; \hat{f}(\xi) \; (M_w g_2)^{\wedge}(\xi)  \; \overline{\hat{h}(\xi)} \;  d\xi \; dw \\
&&= \int_{\T^n} \int_{\T^n} T_w \hat{\tilde{g_1}}(\xi) \; \hat{f}(\xi) \; T_w \hat{g_2}(\xi) \; \overline{\hat{h}(\xi)} \;  d\xi \; dw \\
&&= \int_{\T^n} \left( \int_{\T^n} \overline{\hat{g_1}(\xi-w)} \; \hat{g_2}(\xi-w) \; dw \right) \; \hat{f}(\xi) \; \overline{\hat{h}(\xi)} \;  d\xi. 
\end{eqnarray*}
In the following,  we obtain an $\ell^1$-estimate for the paraproduct $p_{g_1, g_2}(f,h)$ in terms of the $\ell^2$-norms of $f$, $h$, $g_1$ and $g_2$. From Proposition \ref{pro01}, we know that localization operators $\mathfrak{L}^{g_1, g_2}_{\varsigma}: \ell^2(\Z^n) \to \ell^2(\Z^n)$ associated to symbols $\varsigma \in L^\infty(\Z^n \times \T^n)$  are bounded linear operators such that
\[ \left\Vert \mathfrak{L}^{g_1, g_2}_{\varsigma} \right\Vert_{\mathcal{B}(\ell^2(\Z^n))} \leq \Vert \varsigma \Vert_{L^\infty(\Z^n \times \T^n)} \; \Vert g_1 \Vert_{\ell^2(\Z^n)} \; \Vert g_2 \Vert_{\ell^2(\Z^n)}. \]
Hence
\begin{eqnarray*}
\left| \left\langle \mathfrak{L}^{g_1, g_2}_{\varsigma}f, h \right\rangle_{\ell^2(\Z^n)} \right| 
& = & \left|\sum_{k \in \Z^n} \alpha(k)\; p_{g_1, g_2}(f,h)(k) \right| \\
&\leq & \Vert \alpha \Vert_{\ell^\infty(\Z^n)} \; \Vert g_1 \Vert_{\ell^2(\Z^n)} \; \Vert g_2 \Vert_{\ell^2(\Z^n)}\; \Vert f \Vert_{\ell^2(\Z^n)} \; \Vert h \Vert_{\ell^2(\Z^n)} . 
\end{eqnarray*}
Since $p_{g_1, g_2}(f,h) \in \ell^1(\Z^n)$, it follows from the Hahn--Banach theorem that $p_{g_1, g_2}(f,h)$ is in the dual $(\ell^\infty(\Z^n))^*$ of $\ell^\infty(\Z^n)$  and 
\[ \Vert p_{g_1, g_2}(f,h) \Vert_{\ell^1(\Z^n)} \leq \Vert g_1 \Vert_{\ell^2(\Z^n)} \; \Vert g_2 \Vert_{\ell^2(\Z^n)}\; \Vert f \Vert_{\ell^2(\Z^n)} \; \Vert h \Vert_{\ell^2(\Z^n)} .  \]

\subsection{Fourier Multipliers}\label{sec8}

Let $\varsigma$ be a function on $\Z^n \times \T^n$ given by
\[ \varsigma(m, w)= \beta(w),  \quad (m, w) \in \Z^n \times \T^n. \]
Then, by (\ref{eq21}), we obtain for all $f, h \in \ell^2(\Z^n)$, 
\begin{eqnarray*}
&& \left\langle \mathfrak{L}^{g_1, g_2}_{\varsigma}f, h \right\rangle_{\ell^2(\Z^n)} \\
&& = \sum_{m \in \Z^n}  \int_{\T^n} \beta(w) \left( \int_{\T^n} \hat{f}(\xi) \;e^{2 \pi i m \cdot \xi} \overline{\hat{g_1}(\xi-w)} \; d\xi \right) \left(\int_{\T^n} \overline{\hat{h}(\eta)} \; e^{-2 \pi i m \cdot \eta} \hat{g_2}(\eta-w) \; d\eta \right) dw \\
&&= \sum_{m \in \Z^n}  \int_{\T^n} \int_{\T^n} \int_{\T^n} \beta(w) \; \hat{f}(\xi) \; e^{2 \pi i m \cdot (\xi - \eta)} \;\overline{\hat{g_1}(\xi-w)} \; \overline{\hat{h}(\eta)} \; \hat{g_2}(\eta-w) \; d\xi \; d\eta \;dw.
\end{eqnarray*}
Let $\phi(m)=e^{-  |m|^2 / 2}$, $m \in \Z^n$. For all positive numbers $\varepsilon$, let $\phi_{\varepsilon}(m)=e^{-  \varepsilon^2 |m|^2 / 2}$ and $I_\varepsilon$ be the number defined by 
\begin{eqnarray*}
I_\varepsilon 
& = & \sum_{m \in \Z^n}  \int_{\T^n} \int_{\T^n} \int_{\T^n} \beta(w) \; \hat{f}(\xi)\; e^{-  \varepsilon^2 |m|^2 / 2} \; e^{2 \pi i m \cdot (\xi - \eta)} \;\overline{\hat{g_1}(\xi-w)} \; \overline{\hat{h}(\eta)} \; \hat{g_2}(\eta-w) \; d\eta \;dw \; d\xi \\
& = & \int_{\T^n} \int_{\T^n} \int_{\T^n} \left( \sum_{m \in \Z^n} e^{-  \varepsilon^2 |m|^2 / 2}  \; e^{-2 \pi i m \cdot (\xi - \eta)} \right)\beta(w) \; \hat{f}(\xi)\; \overline{\hat{g_1}(\xi-w)} \; \overline{\hat{h}(\eta)} \; \hat{g_2}(\eta-w)  \; d\eta \;dw \; d\xi \\
& = & \int_{\T^n} \int_{\T^n} \int_{\T^n} \hat{\phi_{\varepsilon}}(\xi - \eta) \; \beta(w) \; \hat{f}(\xi)\; \overline{\hat{g_1}(\xi-w)} \; \overline{\hat{h}(\eta)} \; \hat{g_2}(\eta-w)  \; d\eta \;dw \; d\xi \\
& = & \int_{\T^n} \int_{\T^n} \left( \int_{\T^n} \hat{\phi_{\varepsilon}}(\xi - \eta) \;(M_w g_2)^{\wedge}(\eta) \; \overline{\hat{h}(\eta)} \; d\eta \right)\beta(w) \; \hat{f}(\xi) \; \overline{(M_w g_1)^{\wedge}(\xi)} \;dw \; d\xi   \\
& = & \int_{\T^n} \int_{\T^n} \beta(w) \; \overline{(M_w g_1)^{\wedge}(\xi)} \;  \hat{f}(\xi) \; (((M_w g_2)^{\wedge} \; \overline{\hat{h}})*\hat{\phi_{\varepsilon}})(\xi) \;dw \; d\xi . 
\end{eqnarray*}
Now, there exists a sequence $\{ \varepsilon_j \}_{j=1}^\infty$ of positive numbers such that $\varepsilon_j \to 0$, as $j \to \infty$ and
\[ ((M_w g_2)^{\wedge} \; \overline{\hat{h}})*\hat{\phi_{\varepsilon_j}} \to (M_w g_2)^{\wedge} \; \overline{\hat{h}} \]
in $ L^2(\T^n)$ and almost everywhere on $\T^n$ as $ j \to \infty$. Thus, 
\[ I_{\varepsilon_j} \to \int_{\T^n} \int_{\T^n} \beta(w) \; \overline{(M_w g_1)^{\wedge}(\xi)} \;  \hat{f}(\xi) \; (M_w g_2)^{\wedge}(\xi) \; \overline{\hat{h}(\xi)} \;dw \; d\xi. \]
Moreover, we can write
\[I_{\varepsilon}= \sum_{m \in \Z^n}  \int_{\T^n} \beta(w)\; e^{-\varepsilon^2 |m|^2 / 2} \; V_{g_1}f(m,w) \;\overline{V_{g_2}h(m,w)}\; dw .\]   
Thus, using Lebesgue's dominated convergence theorem, we obtain that
\[ I_{\varepsilon_j} \to \sum_{m \in \Z^n}  \int_{\T^n} \beta(w)\; V_{g_1}f(m,w) \;\overline{V_{g_2}h(m,w)}\; dw 
=  \left\langle \mathfrak{L}^{g_1, g_2}_{\varsigma}f , h \right\rangle_{\ell^2(\Z^n)} \]
as $j \to \infty$. Hence for all $f, h \in \ell^2(\Z^n)$, 
\[ \left\langle \mathfrak{L}^{g_1, g_2}_{\varsigma}f , h \right\rangle_{\ell^2(\Z^n)} = \left\langle T_\mu f , h \right\rangle_{\ell^2(\Z^n)}, \]
where $T_\mu$ is the Fourier multiplier with symbol $\mu$ given by 
\[ \mu(\xi)= \int_{\T^n} \beta(w) \; \overline{(M_w g_1)^{\wedge}(\xi)} \; (M_w g_2)^{\wedge}(\xi) \;dw, \quad \xi \in \T^n. \]

\section*{Acknowledgments}
The authors are deeply indebted to Prof. M. W. Wong for several fruitful discussions and generous comments. The authors wish to thank the anonymous referees for their valuable comments and suggestions that helped to improve the quality of the paper.

\section*{Conflict of interest}
The authors declare that there is no potential conflict of interest regarding the publication of this article.

\section*{Data Availability}
The authors confirm that the data supporting the findings of this study are available within the article and its supplementary materials.

\section*{ORCID}
{\it Aparajita Dasgupta} https://orcid.org/0000-0001-7093-8158 \\
{\it Anirudha Poria} https://orcid.org/0000-0002-0224-3642

\end{document}